\documentclass[11pt,a4paper]{amsart}
\usepackage{amsmath, amssymb}
\usepackage{amsfonts}
\usepackage{amsthm}
\usepackage{amscd}
\usepackage[pdftex,colorlinks,linkcolor=magenta,citecolor=cyan]{hyperref}
\usepackage[dvipsnames]{xcolor}

\allowdisplaybreaks

\theoremstyle{plain}
\newtheorem{theorem}{{Theorem}}
\newtheorem{lemma}[theorem]{Lemma}
\newtheorem{proposition}[theorem]{Proposition}
\newtheorem{corollary}[theorem]{Corollary}
\newtheorem{notation}[theorem]{Notation}
\newtheorem{remark}[theorem]{Remark}
\newtheorem{definition}[theorem]{Definition}

\newtheorem{hypothesis}{Hypothesis}

\numberwithin{equation}{subsection}
\newcount\colveccount
\newcommand*\colvec[1]{
	\global\colveccount#1
	\begin{pmatrix}
		\colvecnext
	}
	\def\colvecnext#1{
		#1
		\global\advance\colveccount-1
		\ifnum\colveccount>0
		\\
		\expandafter\colvecnext
		\else
	\end{pmatrix}
	\fi
}

\newcommand{\sA}{\mathsf{A}}

\newcommand{\sP}{\mathsf{P}}
\newcommand{\sK}{\mathsf{K}}
\newcommand{\sZ}{\mathsf{Z}}
\newcommand{\sX}{\mathsf{X}}

\newcommand{\sG}{\mathsf{G}}
\newcommand{\sV}{\mathsf{V}}
\newcommand{\sM}{\mathsf{M}}
\newcommand{\sN}{\mathsf{N}}
\newcommand{\sD}{\mathsf{D}}

\newcommand{\tA}{\mathtt{A}}
\newcommand{\tAd}{\mathtt{Ad}}
\newcommand{\tad}{\mathtt{ad}}
\newcommand{\tAf}{\mathtt{Af}}
\newcommand{\tC}{\mathtt{C}}
\newcommand{\tB}{\mathtt{B}}
\newcommand{\tL}{\mathtt{L}}
\newcommand{\tT}{\mathtt{T}}
\newcommand{\tI}{\mathtt{I}}
\newcommand{\tR}{\mathtt{R}}
\newcommand{\tdR}{{\mathtt{dR}}}
\newcommand{\tJd}{\mathtt{Jd}}
\newcommand{\tM}{\mathtt{M}}
\newcommand{\tP}{\mathtt{P}}

\newcommand{\tth}{\mathtt{h}}
\newcommand{\te}{\mathtt{e}}
\newcommand{\tu}{\mathtt{u}}

\newcommand{\R}{\mathbb{R}}
\newcommand{\C}{\mathbb{C}}

\newcommand{\defeq}{\mathrel{\mathop:}=}

\newcommand{\fg}{\mathfrak{g}}
\newcommand{\fa}{\mathfrak{a}}

\newcommand{\fk}{\mathfrak{k}}

\newcommand{\fp}{\mathfrak{p}}
\newcommand{\fn}{\mathfrak{n}}
\newcommand{\fm}{\mathfrak{m}}

\newcommand{\sfG}{\sG\ltimes\fg}
\newcommand{\GV}{\sG\ltimes\sV}
\newcommand{\tr}{\mathsf{tr}}

\renewcommand{\dim}{\mathsf{dim}}
\renewcommand{\det}{\mathsf{det}}
\renewcommand{\exp}{\mathsf{exp}}
\renewcommand{\ker}{\mathsf{ker}}

\title[Margulis--Smilga spacetimes]{Isospectrality of Margulis--Smilga spacetimes for irreducible representations of real split semisimple Lie groups}
\author{Sourav Ghosh}
\address{Universit\' e du Luxembourg \footnote{Current affiliation: Ashoka University}}
\email{sourav.ghosh@ashoka.edu.in, sourav.ghosh.bagui@gmail.com}
\date{\today}

\thanks{The author acknowledges support from OPEN/16/11405402 grant and the Ashoka University annual research grant.}
\begin{document}
	
	\begin{abstract}
		In this article, we look at real split semisimple algebraic groups $\mathsf{G}$ with trivial center and faithful irreducible algebraic representations $\mathtt{R}$ of $\mathsf{G}$ on some vector space $\mathsf{V}$ which admit zero as a weight and which are self-contragredient (for example, adjoint representation of $\mathsf{PSL}(n,\mathbb{R})$).
		
		We show that, there exist polynomials made out of Margulis invariants of $(g,X)\in\mathsf{G}\ltimes_\mathtt{R}\mathsf{V}$ which are also rational expressions in $(g,X)$. Moreover, we show that any Zariski dense finitely generated subgroup of $\mathsf{G}\ltimes_\mathtt{R}\mathsf{V}$, for which the linear parts of the non-identity elements are loxodromic, is isospectrally rigid with respect to the Margulis invariants. In particular, we show that Margulis--Smilga spacetimes are isospectrally rigid too. 
	\end{abstract}
	
	\maketitle
	\tableofcontents
	
	\newpage
	
	\section*{Introduction}
	
	Representations of finitely generated groups inside semisimple algebraic groups is an oft studied object. Indeed, such representations are often fully determined up to conjugacy by some spectral data. Indeed, the case where the target is a semisimple algebraic group has been extensively studied (see \cite{Otal, DK, CDel, DLR, BCL, BCLS, BC}). The objects of interest in this article are representations of finitely generated groups inside affine algebraic groups. We show that, under certain conditions, these affine representations are also fully determined up to conjugacy by the spectral data of \emph{Margulis invariants}. 
	
	We start by describing the objects of our study, then we state the main results and compare our results with existing results. 
	
	\smallskip
	\paragraph{\textbf{Objects of study}}

	Let $\sG$ be a real split semisimple algebraic group with trivial center and $\fg$ be its Lie algebra. Let $\sV$ be a finite dimensional real vector space whose dimension is at least two, and let $\tR:\sG\to\mathsf{GL}(\sV)$ be a faithful irreducible algebraic representation. We consider the semi-direct product $\sG\ltimes_\tR\sV$ and note that it is naturally isomorphic to a subgroup of the affine group $\mathsf{GL}(\sV)\ltimes\sV$. Thus $\sG\ltimes_\tR\sV$ acts on $\sV$ as affine transformations. Let $\tL:\sG\ltimes_\tR\sV\to\sG$ be such that $\tL(g,X)=g$ for all $(g,X)$. We note that $\tL$ is a homomorphism. Let  $\tT:\sG\ltimes_\tR\sV\to\sV$ be such that $\tT(g,X)=X$ for all $(g,X)$. We note that $\tT$ is a cocycle. We call the image under $\tL$ as the \emph{linear part} and the image under $\tT$ as the \emph{translation part}. Let $\Gamma$ be a finitely generated group and $\rho:\Gamma\to\sG\ltimes_\tR\sV$ be an injective homomorphism.
	
	In order to keep the Introduction simple, we consider the case when $\sG$ is $\mathsf{PSL}(n,\R)$, $\sV$ is $\fg$ and $\tR$ is the adjoint representation $\tAd$. In this case, the representation $\rho$ corresponds to a variation of $\tL(\rho)$. More precisely, suppose $\varrho_t:\Gamma\to\mathsf{PSL}(n,\R)$ is an analytic one parameter family of injective homomorphisms for $-1<t<1$. Let $\dot{\varrho}_0:\Gamma\to\fg$ be the map such that $\dot{\varrho}_0(\gamma)=\left.\frac{d}{dt}\right|_{t=0}\varrho_t(\gamma)\varrho_0(\gamma)^{-1}$. Then $(\varrho_0,\dot{\varrho}_0):\Gamma\to\mathsf{PSL}(n,\R)\ltimes_\tAd\mathfrak{sl}(n,\R)$ is an injective homomorphism. 
	
	We say $g\in\mathsf{PSL}(n,\R)$ is \emph{loxodromic} if it is diagonalizable over $\R$ and the modulus of its eigenvalues are distinct. Also, $\varrho:\Gamma\to\mathsf{PSL}(n,\R)$ is called a \emph{loxodromic representation} if $\varrho(\gamma)$ is loxodromic for all non identity elements $\gamma\in\Gamma$. The \emph{Jordan projection} of $g$, denoted by $\tJd(g)$, is the ordered tuple of the modulus of the eigenvalues of $g$. We note that the notion of a Jordan projection extends to more general split algebraic groups (see Section \ref{subsec.jordan} for more details).  
	
	The \emph{Margulis invariant} of $(g,X)\in\mathsf{PSL}(n,\R)\ltimes_\tAd\mathfrak{sl}(n,\R)$, with $g$ loxodromic, is the ``variation'' of the Jordan projection of $g$ along the direction $X$ (see Section \ref{subsec.marginv} for a more careful definition, and see \cite{Ghosh7} and \cite{Samba2} for more details about the variational point of view). More precisely, using the above notation, the Margulis invariant of $(g_0,\dot{g}_0)\in\mathsf{PSL}(n,\R)\ltimes_\tAd\mathfrak{sl}(n,\R)$ is the derivative of $\tJd(g_t)$ at $t=0$. We denote the Margulis invariant of $(g,X)$ by $\tM(g,X)$ and the \emph{Jordan--Margulis invariant} of $(g,X)$ by the tuple $(\tJd(g),\tM(g,X))$. In this article, we study these invariants of the elements of $\rho(\Gamma)$. These invariants are known to aid in understanding the properness of the action of $\rho(\Gamma)$ on $\sV$ (see \cite{Margulis1,Margulis2,AMS,Smilga, Smilga4, DGK3, GLM, DGK1, GT} and Section \ref{sec.appli} for more details). 
	
	Suppose $\rho$ and $\varrho$ are two representations with loxodromic linear parts. We say $\rho$ and $\varrho$ are \emph{isospectral with respect to Margulis invariants} (respectively \emph{Jordan--Margulis invariants}), if $\tM(\varrho(\gamma))=\tM(\rho(\gamma))$ (respectively $(\tJd(\tL(\varrho(\gamma))),\tM(\varrho(\gamma)))=(\tJd(\tL(\rho(\gamma))),\tM(\rho(\gamma)))$ ) for all non identity elements $\gamma\in\Gamma$. Moreover, when $\tR$ preserves a non-degenerate symmetric bilinear form $\tB_\tR$, we say $\rho$ and $\varrho$ are \emph{isospectral with respect to normed Margulis invariants}, if $\tB_\tR(\tM(\varrho(\gamma)),\tM(\varrho(\gamma)))=\tB_\tR(\tM(\rho(\gamma)),\tM(\rho(\gamma)))$ for all non identity elements $\gamma\in\Gamma$. We observe that, if $\rho$ and $\varrho$ are \emph{conjugate} - that is, if there exists some $(h,Y)\in\sG\ltimes_\tR\sV$ such that $\rho=(h,Y)\varrho(h,Y)^{-1}$ - then they are both isospectral and normed isospectral. 
	
	\smallskip
	\paragraph{\textbf{Statements of main results}}
	
	We collect the conditions appearing in the previous paragraphs into one place and add a few more conditions needed to prove our results in the following two Hypothesis:
	
	\begin{hypothesis}\label{hyp.1} 
		We say that $\sG$ satisfies Hypothesis \ref{hyp.1}, if it is a real split semisimple algebraic group with trivial center. We say that $\sV$ satisfies Hypothesis \ref{hyp.1}, if it is a finite dimensional real vector space of dimension at least two. We say that $\tR:\sG\to\mathsf{GL}(\sV)$ satisfies Hypothesis \ref{hyp.1}, if it is a faithful, irreducible, algebraic representation which admits zero as a weight \footnote{We need to assume that $\tR$ admits zero as a weight in order to define the notion of a Margulis invariant in the general case.}.
	\end{hypothesis}
	We note that $(\mathsf{PSL}(n,\R),\mathfrak{sl}(n,\R),\tAd)$ satisfy Hypothesis \ref{hyp.1}. We also note that $\tAd$ preserves a natural non-degenerate symmetric bilinear form called the \emph{Killing form}. We recommend that the reader at a first reading should only consider the case when $(\sG,\sV,\tR)=(\mathsf{PSL}(n,\R),\mathfrak{sl}(n,\R),\tAd)$. 
	
	\begin{hypothesis}\label{hyp.2}
		We say that $\Gamma$ satisfies Hypothesis \ref{hyp.2}, if it is a finitely generated group. We say that an injective homomorphism $\rho:\Gamma\to\sG\ltimes_\tR\sV$ satisfies Hypothesis \ref{hyp.2}, if $\tL(\rho)$ is loxodromic and $\tL(\rho(\Gamma))$ is Zariski dense in $\sG$.
	\end{hypothesis}
	We note that one can use a ``ping-pong'' argument to obtain examples satisfying Hypothesis \ref{hyp.2} (for more details see \cite{Tit2, Ben}).
	
	In this article, we prove the following results:
	
	\begin{theorem}[see Corollary \ref{cor.JM}]\label{thm.isoJM}
		Suppose $(\sG,\sV,\tR)$ satisfy Hypothesis \ref{hyp.1}, with $\mathsf{rank}(\sG)\geq2$, and $\rho,\varrho:\Gamma\to\sG\ltimes_\tR\sV$ both satisfy Hypothesis \ref{hyp.2}. Then, if $\rho$ and $\varrho$ are isospectral with respect to the Jordan--Margulis invariants, they must be conjugate.
	\end{theorem}
	
	Note that in the following results the hypothesis only involves the Margulis invariants and not the Jordan--Margulis invariants.
	
	\begin{theorem}[see Corollary \ref{cor.marlin}]\label{thm.marlin}
		Suppose $(\sG,\sV,\tR)$ satisfy Hypothesis \ref{hyp.1}, and $\rho,\varrho:\Gamma\to\sG\ltimes_\tR\sV$ both satisfy Hypothesis \ref{hyp.2} with $\tL(\rho)=\tL(\varrho)$. Then, if $\rho$ and $\varrho$ are isospectral with respect to Margulis invariants, they must be conjugate.
	\end{theorem}
	
	\begin{theorem}\label{thm.normainintro}(see Theorem \ref{thm.normain} and Corollary \ref{cor.maintech})
		Suppose $(\sG,\sV,\tR)$ satisfy Hypothesis \ref{hyp.1}, with $\tR$ preserving a non-degenerate symmetric bilinear form, and $\rho,\varrho:\Gamma\to\sG\ltimes_\tR\sV$ both satisfy Hypothesis \ref{hyp.2} with $\rho$ Zariski dense. Then, if $\rho$ and $\varrho$ are isospectral with respect to normed Margulis invariants, they must be conjugate via some element of $\mathsf{GL}(\sV)\ltimes\sV$.
	\end{theorem}
	
	In particular, when $(\sG,\sV,\tR)=(\mathsf{PSL}(n,\R),\mathfrak{sl}(n,\R),\tAd)$, we deduce:
	\begin{corollary} \label{cor.sln}
		Any two Zariski dense representations of a finitely generated group $\Gamma$ inside $\mathsf{PSL}(n,\R)\ltimes_\tAd\mathfrak{sl}(n,\R)$, whose linear parts are loxodromic and which are isospectral with respect to Margulis invariants, are conjugate via some element of $\mathsf{GL}(\mathfrak{sl}(n,\R))\ltimes\mathfrak{sl}(n,\R)$.
	\end{corollary}
	
	We note that the above Theorems hold for representations $\rho$ which give rise to \emph{Margulis--Smilga manifolds} too (for more details see Section \ref{sec.appli}).
	
	Finally, the following result help us understand the vanishing of the Margulis invariants. We say $\rho:\Gamma\to\sG\ltimes_\tR\sV$ admits a \emph{global fixed point}, if for some $X\in\sV$, $\rho(\gamma)X=X$ for all $\gamma\in\Gamma$. 
	
	\begin{theorem}\label{cor.tfae.intro}(see Proposition \ref{prop.zd} and Corollary \ref{cor.tfae})
		Suppose $(\sG,\sV,\tR)$ satisfy Hypothesis \ref{hyp.1}, and $\rho:\Gamma\to\sG\ltimes_\tR\sV$ satisfy Hypothesis \ref{hyp.2}. Then the following are equivalent:
		\begin{enumerate}
			\item $\rho$ admits a global fixed point.
			\item $\rho$ is not Zariski dense inside $\sG\ltimes_\tR\sV$.
			\item The Margulis invariant spectrum of $\rho$ is identically zero.
			\item $\rho$ is conjugate to $\tL(\rho)$
		\end{enumerate}
		Moreover, if $\tR$ preserves a non-degenerate symmetric bilinear form, then the above conditions are  equivalent to the normed Margulis invariant spectrum of $\rho$ being identically zero \footnote{As the symmetric bilinear form is not positive definite, non-zero vectors can have zero norms.}.
	\end{theorem}
	
	\smallskip
	\paragraph{\textbf{Comparison with previous papers}}
	
	Our main results (see Section \ref{sec.nir}) develop proofs whose underlying idea is the same as the one used by Kim to prove isospectral rigidity results in \cite{Kim}. It is as follows: equality of the marked Margulis invariant spectra of two Zariski dense representations $\rho,\varrho:\Gamma\to\sG\ltimes_\tR\sV$ forces the Zariski closure of $\{(\rho(\gamma),\varrho(\gamma))\mid\gamma\in\Gamma\}$ to be a graph inside $(\sG\ltimes_\tR\sV)\times(\sG\ltimes_\tR\sV)$, and hence induces an automorphism. This idea works under the assumption that equality of Margulis invariants are essentially algebraic in nature. Although this assumption is stated by Kim \cite{Kim} to obtain his results, yet, in my understanding, there is not enough justification to show algebraicity. In fact, Margulis invariants $\tM(g,X)$ with $(g,X)\in\sG\ltimes_\tR\sV$ are themselves not rational expressions in the variables $(g,X)$. But when the representation of $\sG$ on $\sV$ preserves a non-degenerate symmetric bilinear form, then certain polynomials constructed out of $\tM(g,X)$ are rational expressions in the variables $(g,X)$. The novelty of this article is twofold: 
	\begin{enumerate}
		\item we construct polynomials of $\tM(g,X)$ which are also rational expressions in the variables $(g,X)$,
		\item we generalize results obtained by Drumm--Goldman \cite{DG}, Charette--Drumm \cite{CD} and Kim \cite{Kim} to include a large collection of examples of proper affine actions constructed by Smilga \cite{Smilga4}.
	\end{enumerate}
	
	Moreover, in the appendix, we include a proof of an important statement, which we expect to be known in the community, but were unable to find an appropriate reference for.
	
	\section*{Outline of the article}
	In Section \ref{sec.prelim} we recall certain preliminary results about algebraic groups and their representations. In Section \ref{sec.alg} we prove some results relating characteristic polynomials and projections onto unit eigenspaces to show that for $\tR$ self-contragredient some polynomials made out of Margulis invariants of $(g,X)\in\sG\ltimes_\tR\sV$ are rational expressions in $(g,X)$. In Section \ref{sec.nir} we demonstrate isospectral rigidity of affine actions. Finally, in Section \ref{sec.appli} we apply the results from the previous sections to conclude that conjugacy classes of Margulis--Smilga spacetimes coming from self-contragredient representations are uniquely determined by their Margulis invariant spectra. In general, Margulis--Smilga spacetimes with a fixed linear part are determined by their Margulis invariant spectra. We also show that for $\sG$ with rank atleast two, the Jordan--Margulis spectra of a Margulis--Smilga spacetime uniquely determines them up to conjugacy class.
	
	\section*{Acknowledgements}
	I thank Ilia Smilga for explaining to me his work in \cite{Smilga4} and Arghya Mondal for helpful discussions. I also thank Gregory Soifer for correcting the reference. Finally, I warmly thank Fran\c cois Labourie, Indira Chatterji, Saujanya Bharadwaj and the anonymous referee for constructive inputs which helped me improve the exposition of this article.

	\numberwithin{theorem}{subsection}
	\section{Preliminaries}\label{sec.prelim}
	In this section we recall certain basic results which will be used later in the article to obtain our main results.
	
	\subsection{Jordan decomposition}\label{subsec.jordan}
	Let $\sG$ be a noncompact real semisimple algebraic Lie group with trivial center and let $\fg$ be its Lie algebra. We denote the identity element of $\sG$ by $e$. Let $\tC_g$ be the conjugation map on $\sG$ i.e. for any $g,h\in \sG$ we have $\tC_g(h)=ghg^{-1}$ and let $\tAd_g$ be the differential of this at identity. Hence we obtain a homomorphism $\tAd: \sG \to \mathsf{SL}(\fg)$. Moreover, let $\tad$ be the differential of $\tAd$ at the identity element. We fix a Cartan involution $\theta:\fg\to\fg$ and consider the corresponding decomposition $\fg=\fk\oplus\fp$ where $\fk$ (respectively $\fp$) is the eigenspace of eigenvalue 1 (respectively -1). Let $\fa$ be a maximal abelian subspace of $\fp$. We denote the space of linear forms on $\fa$ by $\fa^*$ and for all $\alpha\in\fa^*$ we define
	\[\fg^\alpha:=\{X\in\fg\mid \tad_H(X)=\alpha(H)X\text{ for all } H\in\fa\}.\]
	We call $\alpha\in\fa^*$ a \textit{restricted root} if and only if both $\alpha\neq0$ and $\fg^\alpha\neq0$. Let $\Sigma\subset\fa^*$ be the set of all restricted roots. As $\fg$ is finite dimensional, it follows that $\Sigma$ is finite. Moreover, we note that
	\[\fg=\fg^0\oplus\bigoplus_{\alpha\in\Sigma}\fg^\alpha.\]
	We choose $\fa^{++}$, a connected component of $\fa\setminus\cup_{\alpha\in\Sigma}\ker(\alpha)$ and denote its closure by $\fa^+$. Let $\sK\subset \sG$ (respectively $\sA\subset \sG$) be the connected subgroup whose Lie algebra is $\fk$ (respectively $\fa$) and let $\sA^+:=\exp{(\fa^+)}$. We note that $\sK$ is a maximal compact subgroup of $\sG$.
	
	Let $\tB$ be the Killing form on $\fg$ i.e. for any $X,Y\in\fg$ we have $\tB(X,Y):=\tr(\tad_X\circ \tad_Y)$. We define $\Sigma^+\subset \Sigma$ to be the set of restricted roots which take positive values on $\fa^+$ and note that $\Sigma=\Sigma^+\sqcup-\Sigma^+$. We consider the following nilpotent subalgebras:
	\[\fn^\pm:=\bigoplus_{\pm\alpha\in\Sigma^+}\fg^\alpha .\]
	Let $\sN$ be the Lie subgroup of $\sG$ generated by $\fn^+$ and $g\in \sG$. Then
	\begin{enumerate}
		\item $g$ is called \textit{elliptic} if and only if some conjugate of $g$ lies in $\sK$,
		\item $g$ is called \textit{hyperbolic} if and only if some conjugate of $g$ lies in $\sA$,
		\item $g$ is called \textit{unipotent} if and only if some conjugate of $g$ lies in $\sN$.
	\end{enumerate}
	\begin{theorem}[Jordan decomposition]
		Suppose $\sG$ is a noncompact real semisimple algebraic Lie group with trivial center. Then for any $g\in \sG$, there exist unique $g_\te,g_\tth,g_\tu\in \sG$ such that the following hold:
		\begin{enumerate}
			\item $g=g_\te g_\tth g_\tu$,
			\item $g_\te$ is elliptic, $g_\tth$ is hyperbolic and $g_\tu$ is unipotent,
			\item the elements $g_\te,g_\tth,g_\tu$ commute with each other.
		\end{enumerate}
	\end{theorem}
	\begin{definition}\label{def.jordan}
		Let $\sG$ be a noncompact real semisimple algebraic Lie group with trivial center and let $g\in \sG$. Then the \textit{Jordan projection} of $g$, denoted by $\tJd_g$, is the unique element in $\fa^+$ such that $g_\tth$ is a conjugate of $\exp{(\tJd_g)}$.
	\end{definition}
	\begin{remark}\label{rem.jdcont}
		We note that $\tJd$ is continuous. Indeed, we use Lemmas 6.32 and 6.33 (ii) of \cite{BQ} and Appendix V.4 of \cite{Whit} to deduce it (see also \cite{Tit}).
	\end{remark}
	\begin{definition}
		Let $\sG$ be a noncompact real semisimple algebraic Lie group with trivial center and let $g\in \sG$. Then $g$ is called \textit{loxodromic} if and only if $\tJd_g\in\fa^{++}$.
	\end{definition}
	Moreover, let $\sM$ be the centralizer of $\fa$ inside $\sK$ and $\fm$ be the Lie subalgebra of $\fg$ coming from $\sM$. We note that $\fg=\fn^+\oplus\fg^0\oplus\fn^-$ and $\fg^0=\fa\oplus\fm$.
	\begin{remark}
		If $\sG$ is split then $\fm$ is trivial and $\sM$ is a finite group (see Theorem 7.53 of \cite{Knapp}).
	\end{remark}
	
	\begin{proposition}[see e.g. Proposition 2.31 of \cite{Thi}] \label{prop.dang}
		Let $\sG$ be a noncompact real semisimple algebraic Lie group with trivial center and let $g\in \sG$ be loxodromic. Then the following hold:
		\begin{enumerate}
			\item $g_\tu$ is trivial,
			\item for $h_g\in \sG$ with $g_\tth=h_g\exp(\tJd_g)h_g^{-1}$, we have $m_g:=h_g^{-1}g_\te h_g\in \sM$,
			\item for $(h_g,m_g)\in \sG\times \sM$ as above, we have $g=h_gm_g\exp(\tJd_g)h_g^{-1}$,
			\item if $(h,m)\in \sG\times \sM$ satisfies $g=hm\exp(\tJd_g)h^{-1}$, then the element $c:=h_g^{-1}h\in \sM\sA$ and $m=c^{-1}m_gc$.
		\end{enumerate}
	\end{proposition}
	We observe that the Jordan projection are invariant under conjugation, i.e. for all $g,h\in\sG$ we have $\tJd_{hgh^{-1}}=\tJd_g$.
	
	\subsection{Weights and eigenspaces}
	Let $\sG$ be a noncompact real semisimple Lie group with trivial center, $\sV$ be a finite dimensional vector space and let $\tR:\sG\to\mathsf{GL}(\sV)$ be a faithful irreducible representation. In order to define the appropriate notion of a Margulis invariant in the general setting we need to assume that the dimension of the unit eigenspace of $\tR(g)$ for any loxodromic element $g\in\sG$ is at least $1$. The technical condition needed to guarantee such behaviour is that $\tR$ admits zero as a \emph{weight}.

	We consider the differential $\tdR$ of the representation $\tR$ at the identity, and obtain a Lie algebra representation $\tdR:\fg\to\mathfrak{gl}(\sV)$. We observe that for any $X\in\fa$ and $t\in\R$, the following holds: $\tR(\exp(tX))=\exp(t\tdR(X))$. We recall that $\fa^*$ denotes the space of all linear forms on $\fa$ and for all $\lambda\in\fa^*$ we define
	\[\sV^\lambda:=\{X\in\sV\mid \tdR_H(X)=\lambda(H)X\text{ for all } H\in\fa\}.\] 
	We call $\lambda\in\fa^*$ a \textit{restricted weight} of the representation $\tR$ if and only if $\sV^\lambda\neq0$. Let $\Omega\subset\fa^*$ be the set of all restricted weights. As $\sV$ is finite dimensional, it follows that $\Omega$ is finite. Moreover, we note that
	\[\sV=\bigoplus_{\lambda\in\Omega}\sV^\lambda.\]
	\begin{notation}
		Henceforth we will also use the expression $\tR_g$ to denote $\tR(g)$ for any $g\in\sG$.
	\end{notation}
	
	\begin{remark}\label{rem.weight}
		Let $g\in\sA$ and let $X_g\in\fa$ be such that $g=\exp(X_g)$. Then for any $X\in\sV^\lambda$ we have $\tR_g X=\exp(\lambda(X_g))X$. In particular, 
		\[\sV^\lambda=\{X\in\sV\mid \tR_g(X)=\exp(\lambda(X_g))X\text{ for all } g\in\sA\}.\]
		
		Henceforth, we will only consider representations $\tR$ such that $\sV^0$ is nontrivial and we will denote $\bigoplus_{\lambda\in\Omega\setminus\{0\}}\sV^\lambda$ by $\sV^{\neq0}$.
	\end{remark}
	
	\begin{lemma}\label{lem.eigen}
		Suppose $\sG$ satisfy Hypothesis \ref{hyp.1}. Then for any $g\in\sM\sA$ and any $\lambda\in\fa^*$ we have $\tR_g \sV^\lambda=\sV^\lambda$. Moreover, for $\lambda=0$ we have
		\[\sV^0=\{X\in\sV\mid \tR_g(X)=X\text{ for all } g\in\sM\sA\}.\]
	\end{lemma}
	\begin{proof}
		As $\sG$ is split we have $\fm=0$. Hence $\sM$ is a discrete group. It follows that the connected component of $\sM$ containing identity is a singleton. Now we use Theorem 7.53 of \cite{Knapp} and observe that $\sM\subset\exp(i\fa)$. Hence for any $m\in\sM$ there exists $X_m\in\fa$ such that $m=\exp(iX_m)$. It follows that for any $X\in\sV^\lambda$, we have 
		\begin{equation}\label{eq.exp}
			\tR_mX=\exp(i\tdR(X_m))X=\exp(i\lambda(X_m))X.
		\end{equation}
		Also, as $X$ and $\tR_mX$ lie inside the real part, we obtain that $\exp(i\lambda(X_m))\in\R$. Therefore, $\tR_mX\in\sV^\lambda$ if and only if $X\in\sV^\lambda$. 
		
		Finally, using the definition of $\sV^0$ we obtain that
		\[\sV^0\supset\{X\in\sV\mid \tR_g(X)=X\text{ for all } g\in\sM\sA\}.\] Also, using Equation \ref{eq.exp}, for $X\in\sV^0$, we get
		$\tR_mX=\exp(0)X=X$, and using Remark \ref{rem.weight} we obtain that
		$\tR_gX=X$ for all $g\in\sM\sA$. Hence, 
		\[\sV^0\subset\{X\in\sV\mid \tR_g(X)=X\text{ for all } g\in\sM\sA\}\]
		and we conclude our result.
	\end{proof}
	
	\begin{lemma}\label{lem.ue}
		Suppose $\sG$ satisfy Hypothesis \ref{hyp.1} and let $g\in\sG$ be a loxodromic element. Then the dimension of the unit eigenspace of $\tR_g$ is at least $\dim \sV^0$. Moreover, the set of loxodromic elements $h\in\sG$ such that the dimension of the unit eigenspace of $\tR_h$ is exactly $\dim \sV^0$, is a non-empty open dense subset of $\sG$.
	\end{lemma}
	\begin{proof}
		We use Proposition \ref{prop.dang} and Lemma \ref{lem.eigen} to deduce that the unit eigenspace of $\tR_g$ for a loxodromic element $g$ is at least $\dim \sV^0$.
		
		Suppose $Y\in\fa$ is such that $\alpha(Y)\neq0$ for all $\alpha\in\Sigma$ and $\lambda(Y)\neq0$ for all $\lambda\in\Omega$, then by Remark \ref{rem.weight}, we get that $\tR(\exp(Y))X=\exp(\lambda(Y))X$. We again use Theorem 7.53 of \cite{Knapp} and the property that $\sG$ is split to conclude that for any $m\in\sM$ there exists $X_m\in\fa$ such that $m=\exp(iX_m)$ and
		\[\tR_mX=\exp(i\lambda(X_m))X\]
		with $\exp(i\lambda(X_m))\in\R$. Hence, $\exp(i\lambda(X_m))=\pm1$, and for any $X\in\sV^\lambda$ with $\lambda\neq0$ we have
		\[\tR(m\exp(Y))X=\exp(\lambda(Y))\tR_mX=\pm\exp(\lambda(Y))X\neq X.\]
		Also, using Lemma \ref{lem.eigen} we obtain that $\tR_g(X)=X$  for all $g\in\sM\sA$ and for all $X\in\sV^0$.
		Therefore, our result follows by using Remark \ref{rem.jdcont}, Proposition \ref{prop.dang} and observing that the set 
		\[\fa\setminus\left(\bigcup_{\alpha\in\Sigma}\ker(\alpha)\cup\bigcup_{\lambda\in\Omega}\ker(\lambda)\right)\]
		is a non-empty open dense subset of $\fa$, as the sets $\Sigma$ and $\Omega$ are finite.
	\end{proof}
	
	We note that the action of $\tAd(\sG)\subset\mathsf{GL}(\mathfrak{sl}(n,\R))\subset\mathsf{GL}(\mathfrak{sl}(n,\C))$ is also irreducible on $\mathfrak{sl}(n,\C)$. More generally, as $\sG$ is real split and $\tR$ is irreducible, using Proposition 26.23 of \cite{FultonHarris} and the tables contained in Reference Chapter 1.3 of \cite{OniVin} (see also Appendix B of \cite{RanCui} for a corrected version) we note that the action of $\tR(\sG)\subset\mathsf{GL}(\sV)\subset\mathsf{GL}(\sV\otimes\C)$ is also irreducible on $\sV\otimes\C$. Such a representation is called \emph{absolutely irreducible}. 
	
	Moreover, the adjoint representation of $\mathsf{PSL}(n,\R)$ on $\mathfrak{sl}(n,\R)$ preserves a natural symmetric bilinear form $\tB_n$ called the \emph{Killing form}. Explicitly, for any $X,Y\in\mathfrak{sl}(n,\R)$, $\tB_n(X,Y)=2n\tr(XY)$. In general, there is no guarantee that $\tR$ preserves a bilinear form on $\sV$. The technical condition needed to guarantee the existence of such a bilinear form on $\sV$ is called self-contragredience. We call a representation $\tR$ \textit{self-contragredient} if $\tR$ is conjugate to its dual $(\tR^t)^{-1}$ (for more details see Section 3.11 of \cite{Samel}). Further, since by our assumption, zero is a weight of $\tR$, by Lemma 1.3 of \cite{Gros1}, these self-contragredient representations $\tR$ admit an invariant \emph{symmetric} bilinear form $\tB_\tR$. Hence, $\sG\subset\mathsf{O}(\tB_\tR)\subset\mathsf{GL}(\sV)$. We note that the signature of this symmetric bilinear form $\tB_\tR$ is $(p,q)$ where $p,q$ are both non-zero. Henceforth, we denote $\tB_\tR(X,X)$ by $\tB_\tR(X)$ for notational simplicity and call it the \emph{norm} of $X$.
	
	\section{Algebraicity}\label{sec.alg}
	In this section we define Margulis invariants $\tM(g,X)$ for $g\in\sG$ and $X\in\sV$. Moreover, we show that Margulis invariants are themselves not rational expressions in the variables $(g,X)$, but when the representation of $\sG$ on $\sV$ admits an invariant symmetric bilinear form, then polynomials constructed out of $\tM(g,X)$ are rational expressions in the variables $(g,X)$.

	\subsection{Characteristic polynomial}
	Let $\tA\in\mathfrak{gl}(\sV)$ and let $\tI\in\mathfrak{gl}(\sV)$ be the diagonal matrix with all its diagonal entries equal to $1$. Then the \emph{characteristic polynomial} of $\tA$ in the indeterminate $x$ is defined by the following expression:
	\[\chi_{\tA}(x)\defeq\det(x\tI-\tA).\]
	\begin{theorem}[Cayley--Hamilton, see \cite{FCH}]\label{thm.ch}
		Let $\tA\in\mathfrak{gl}(\sV)$. Then \[\chi_{\tA}(\tA)=0.\]
	\end{theorem}
	Suppose $(\sG,\sV,\tR)$ satisfy Hypothesis \ref{hyp.1} and let $g\in\sG$. We alternately denote $\tR(g)$ by $\tR_g$. Then $\tR_g\in\mathsf{GL}(\sV)\subset\mathfrak{gl}(\sV)$. Hence $(\tR_e-\tR_g)\in\mathfrak{gl}(\sV)$. Also we observe that $\tR_e=\tI$. 
	
	\begin{remark}
		One can use the Cayley--Hamilton Theorem \ref{thm.ch} to deduce that the characteristic polynomial of $(\tR_e-\tR_g)$ has the following expression:
		\[\sum_{k=0}^{\dim\sV}(-1)^{\dim\sV-k}\tr(\wedge^{\dim\sV-k}(\tR_e-\tR_g))x^k.\]
		Hence the coefficients of the characteristic polynomial are also algebraic.
	\end{remark}
	
	Let $\R[x]$ be set of all polynomials in the indeterminate $x$ and note that $\R[x]$ is a principal ideal domain i.e. any ideal is generated by a single polynomial, which is unique up to units in $\R[x]$. Now for $\tA\in\mathfrak{gl}(\sV)$ we consider:
	\[\mathsf{I}_\tA:=\{p(x)\in\R[x]\mid p(\tA)=0\}\]
	and observe that $\mathsf{I}_\tA$ is a proper ideal of $\R[x]$. 
	
	\begin{definition}
		The minimal polynomial of $\tA\in\mathfrak{gl}(\sV)$ is the unique monic polynomial which generates $\mathsf{I}_\tA$. It is the monic polynomial of least degree inside $\mathsf{I}_\tA$.
	\end{definition}

	\begin{remark}\label{rem.mp}
		We observe that by definition the minimal polynomial of $\tA$ divides the characteristic polynomial of $\tA$. In fact, one can deduce that the minimal polynomial of $\tA$ and the characteristic polynomial of $\tA$ have the same irreducible factors in $\R[x]$. Moreover, when $g$ is loxodromic, using Proposition 14 in Chapter 7.5.8 of \cite{B4-7} we get that the minimal polynomial of $(\tR_e-\tR_g)$ has no multiple factors. Equivalently, $\tR_g$ is diagonalizable over $\C$ when $g\in\sG$ is loxodromic (see also Theorem 2.4.8 (ii) of \cite{Sp}).
	\end{remark}
	
	\begin{proposition} \label{prop.p}
		Let $g\in\sG$ be loxodromic. Then $x^{\dim\sV^0}$ divides the characteristic polynomial of $(\tR_e-\tR_g)$. We denote the quotient polynomial by $\tP_g(x)$.
		
		Moreover, let $\tP_g(x)=\sum_{k=0}^{\dim(\sV)-\dim(\sV^0)}a_k(g)x^k$. Then the coefficient $a_k(g)$, for any $k\in\{0,1,..,\dim\sV\}$, is algebraic in $g$.
	\end{proposition}
	\begin{proof}
		Using Lemma \ref{lem.ue} we obtain that $\tR_g$ has an eigenvalue $1$ with multiplicity at least $\dim\sV^0$. Hence, $(\tR_e-\tR_g)$ has at least $\dim\sV^0$ many $0$ as eigenvalues. As $\tR$ is algebraic, we conclude by observing that \[a_k(g)=(-1)^{n-k-\dim\sV^0}\tr(\wedge^{n-k-\dim\sV^0}(\tR_e-\tR_g))\] is algebraic in $g$ for all $k$.
	\end{proof}
	
	\begin{remark}\label{rem.chired}
		Suppose $g\in\sG$, $n_0:=\dim\sV^0$, $n:=\dim\sV$  and $0\leq k\leq n-n_0=\dim(\sV^{\neq0})$. We consider $a_k(g):=(-1)^{n-n_0-k}\tr(\wedge^{n-n_0-k}(\tR_e-\tR_g))$ and define
		\[\chi^{red}_g(x):=\sum_{k=0}^{\dim(\sV^{\neq0})}a_k(g)(1-x)^k.\]
		We note that for any loxodromic element $g\in\sG$ we have $\chi^{red}_g(x)=\tP_g(1-x)$ and $\chi^{red}_e(x)=(1-x)^{n-n_0}$. Clearly $\chi^{red}_e(\tR_e)=0$.
	\end{remark}
	
	\begin{lemma}\label{lem.reg}
		Let $g\in\sG$ be a loxodromic element such that the dimension of the unit eigenspace of $\tR_g$ is exactly $\dim \sV^0$. Then $\chi^{red}_g(1)\neq0$.
	\end{lemma}
	\begin{proof}
		We use Remark \ref{rem.mp} and the property that $\sV=\sV^0\oplus\sV^{\neq0}$ to conclude our result.
	\end{proof}

	\subsection{Unit eigenspace projection}
	Let $\pi_0$ be the projection onto the $\sV^0$ component with respect to the decomposition: $\sV=\sV^0\oplus\sV^{\neq0}$.
	
	\begin{lemma}\label{lem.id}
		Let $g\in\sG$ be a loxodromic element and let $\chi^{red}_g$ be as in Remark \ref{rem.chired}. Then, \[(\tR_e-\tR_g)\chi^{red}_g(\tR_g)=0.\]
	\end{lemma}
	\begin{proof}
		As $g$ is loxodromic, using Remark \ref{rem.mp}, we obtain that $\tR_g$ is diagonalizable over $\C$. It follows that the minimal polynomial of $(\tR_e-\tR_g)$ is a product of distinct monic linear factors and hence is divisible by $x$ but not by $x^2$ (See Remark \ref{rem.mp}). Also the minimal polynomial of $(\tR_e-\tR_g)$ divides the characteristic polynomial of $(\tR_e-\tR_g)$ and hence the minimal polynomial of $(\tR_e-\tR_g)$ divides $x\tP_g(x)$. We also know that the minimal polynomial of $(\tR_e-\tR_g)$ evaluated at $(\tR_e-\tR_g)$ is $0$ and it follows that $(\tR_e-\tR_g)\chi^{red}_g(\tR_g)=0$. 
	\end{proof}
	\begin{proposition}\label{prop.parpro}
		Let $c\in\sM\sA$ and let $\chi^{red}_c$ be as in Remark \ref{rem.chired}. Then, for any $X\in\sV$:
		\[\chi^{red}_c(\tR_c)X=\chi^{red}_c(1)\pi_0(X).\]
	\end{proposition}
	\begin{proof}
		As $\tR_c$ is diagonalizable over complex numbers, $\tR_e-\tR_c$ is also diagonalizable over complex numbers. Moreover, as $\tR_cZ=Z$ for all $Z\in\sV^0$, we have $\sV^0\subset\ker(\tR_e-\tR_c)$. We will prove our result in two separate cases:
		\begin{enumerate}
			\item $\sV^0\neq\ker(\tR_e-\tR_c)$: In this case, $\tP_c(x)$ is divisible by $x$. Hence $\tP_c(x)$ is divisible by the minimal polynomial of $(\tR_e-\tR_c)$. Moreover, the minimal polynomial of $(\tR_e-\tR_c)$ evaluated at $(\tR_e-\tR_c)$ is $0$ and we obtain that $\chi^{red}_c(\tR_c)=0$. Also, as $\tP_c(x)$ is divisible by $x$, we have $\chi^{red}_c(1)=0$. Therefore, 
			\[\chi^{red}_c(\tR_c)X=0=\chi^{red}_c(1)\pi_0(X).\]
			\item $\sV^0=\ker(\tR_e-\tR_c)$: In this case, given any $Y\in\sV^{\neq0}$, there exists a $Y^\prime\in\sV^{\neq0}$ such that $(\tR_e-\tR_c)Y^\prime=Y$. Indeed, as $\tR_c\sV^{\neq0}\subset\sV^{\neq0}$ we obtain that \[(\tR_e-\tR_c):\sV^{\neq0}\to\sV^{\neq0}\] 
			is a linear map with kernel $\sV^0\cap\sV^{\neq0}=\{0\}$, and hence $(\tR_e-\tR_c)$ is invertible on $\sV^{\neq0}$. Therefore, for any $X\in\sV$, there exists $Y\in\sV^{\neq0}$ such that $(X-\pi_0(X))=(\tR_e-\tR_c)Y$. It follows that \[\chi^{red}_c(\tR_c)(X-\pi_0(X))=\chi^{red}_c(\tR_c)(\tR_e-\tR_c)Y=0\] 
			and hence, for any $X\in\sV$ we have $\chi^{red}_c(\tR_c)X=\chi^{red}_c(\tR_c)\pi_0(X)$. Moreover, as $(\tR_e-\tR_c)\pi_0(X)=0$, we conclude by observing that 
			\begin{align*}
				\chi^{red}_c(\tR_c)\pi_0(X)&=\sum_{k=0}^{\dim(\sV^{\neq0})}a_k(c)(\tR_e-\tR_c)^k\pi_0(X)\\
				&=a_0(c)\pi_0(X)=\chi^{red}_c(1)\pi_0(X).
			\end{align*}
		\end{enumerate}
		Our result is complete.
	\end{proof}
	
	\begin{proposition}
		Let $g\in\sG$ be a loxodromic element such that the dimension of the unit eigenspace of $\tR_g$ is exactly $\dim \sV^0$. Then, the map \[\chi^{red}_g(1)^{-1}\chi^{red}_g(\tR_g):\sV\to\sV\] 
		is the projection onto the unit eigenspace of $\tR_g$ with respect to the eigenspace decomposition of $\tR_g$.
	\end{proposition}
	\begin{proof}
		We use Lemma \ref{lem.reg} and observe that $\tP_g(0)\neq0$. Hence the map $\chi^{red}_g(1)^{-1}\chi^{red}_g(\tR_g):\sV\to\sV$ is a well defined linear map. Moreover, as $g$ is loxodromic, we use Proposition \ref{prop.dang} and obtain that there exists $h\in\sG$ such that $c:=h^{-1}gh\in\sM\sA$. Now we use Proposition \ref{prop.parpro} and obtain that 
		\[\chi^{red}_c(1)^{-1}\chi^{red}_c(\tR_c)X=\pi_0(X)\]
		for all $X\in\sV$. Hence, $\chi^{red}_c(1)^{-1}\chi^{red}_c(\tR_c)=\pi_0$ is a projection operator projecting onto $\sV^0$. Moreover, as $c=h^{-1}gh$, we deduce that
		\[\chi^{red}_g(1)^{-1}\chi^{red}_g(\tR_g)=\tR_h\circ\pi_0\circ\tR_h^{-1}.\]
		It follows that $\chi^{red}_g(1)^{-1}\chi^{red}_g(\tR_g)$ is a projection onto the space $\tR_h\sV^0$. Therefore, we will be done once we show that $\tR_h\sV^0$ is the unit eigenspace of $\tR_g$. Finally, we observe that $\tR_cX=X$ if and only if $\tR_g\tR_hX=\tR_hX$ and conclude our result using Lemmas \ref{lem.eigen} and \ref{lem.ue}.
	\end{proof}

	\subsection{Margulis invariant}\label{subsec.marginv}
	
	In this subsection, we define the Margulis invariants corresponding to a faithful irreducible representation of a real split semisimple Lie group with trivial center. We also relate these invariants with the unit eigenspace projections introduced in the previous section. Henceforth, when there is no confusion, we will omit the subscript $\tR$ from $\sG\ltimes_\tR\sV$. We denote the affine action of $\GV$ on $\sV$ by $\tAf$ i.e. for any $(g,X)\in\GV$ and $Y\in\sV$ we have:
	\[\tAf_{(g,X)}Y:=\tR_gY+X.\]
	\begin{lemma}
		Let $(g,X)\in\GV$ be such that $g$ is loxodromic and $g_\tth$ be its hyperbolic part with respect to the Jordan decomposition. Let $h_1,h_2$ be such that
		\[h_1\exp(\tJd_g)h_1^{-1}=g_\tth=h_2\exp(\tJd_g)h_2^{-1}.\] 
		Then $\pi_0(\tR_{h_1}^{-1}X)=\pi_0(\tR_{h_2}^{-1}X)$.
	\end{lemma}
	\begin{proof}
		We recall that by Lemma \ref{lem.eigen}, for any $c\in\sM\sA$, we have $\tR_c\sV^{\neq0}=\sV^{\neq0}$ and $\tR_cX=X$ for any $X\in\sV^0$. 
		
		Also by Proposition \ref{prop.dang} there exist some $c\in\sM\sA$ such that $h_2=h_1c$. For $i\in\{1,2\}$, we denote the component of $\tR_{h_i}^{-1}X$ inside $\sV^{\neq0}$ by $Y_i$ and the component of $\tR_{h_i}^{-1}X$ inside $\sV^0$ by $Z_i$. As $h_2=h_1c$ we deduce that $(Y_1+Z_1)=\tR_c(Y_2+Z_2)$, and hence, \[\tR_cY_2-Y_1=Z_1-\tR_cZ_2=Z_1-Z_2.\] 
		We notice that $\left(\tR_cY_2-Y_1\right)\in\sV^{\neq0}$, $\left(Z_1-Z_2\right)\in\sV^0$ and $\sV^0\cap\sV^{\neq0}=\{0\}$. Therefore, $Z_1=Z_2$ and we conclude that $\pi_0(\tR_{h_1}^{-1}X)=\pi_0(\tR_{h_2}^{-1}X)$.
	\end{proof}
	\begin{lemma}
		Let $(g,X)\in\GV$ be such that $g$ is loxodromic and let $g_\tth$ be its hyperbolic part with respect to the Jordan decomposition. Let $h\in\sG$ be such that $g_\tth=h\exp(\tJd_g)h^{-1}$. Then, for any $Y\in\sV$ we have,
		\[\pi_0(\tR_h^{-1}(\tAf_{(g,X)}Y-Y))=\pi_0(\tR_h^{-1}X).\]
	\end{lemma}
	\begin{proof}
		We recall from Proposition \ref{prop.dang} that $c:=h^{-1}gh\in\sM\sA$. We denote $\tR_h^{-1}Y$ by $Z$ and observe that 
		\[Z=\pi_0(Z)+(Z-\pi_0(Z)),\]
		with $\pi_0(Z)\in\sV^0$ and $(Z-\pi_0(Z))\in\sV^{\neq0}$. Hence, using Lemma \ref{lem.eigen} we obtain
		\[\tR_c Z=\pi_0(Z)+\tR_c(Z-\pi_0(Z)),\] 
		with $\tR_c(Z-\pi_0(Z))\in\sV^{\neq0}$. Hence, we deduce that
		\begin{align*}
			\pi_0(\tR_h^{-1}(\tAf_{(g,X)}Y-Y))&=\pi_0(\tR_h^{-1}(\tR_gY-Y))+\pi_0(\tR_h^{-1}X)\\
			&=\pi_0(\tR_cZ-Z)+\pi_0(\tR_h^{-1}X)=\pi_0(\tR_h^{-1}X).
		\end{align*}
	\end{proof}
	\begin{definition}\label{def.msi}
		Let $(g,X)\in\GV$ be such that $g$ is loxodromic, and let $g_\tth$ be its hyperbolic part with respect to the Jordan decomposition. Let $h\in\sG$ be such that $g_\tth=h\exp(\tJd_g)h^{-1}$. Then the \textit{Margulis--Smilga} invariant of $(g,X)$, denoted by $\tM(g,X)$, is defined as follows:
		\[\tM(g,X):=\pi_0(\tR_{h}^{-1}X).\]
	\end{definition}
	\begin{remark}
		Note that by Definition 6.2 of \cite{Smilga4}, Proposition 7.8 of \cite{Smilga4}, and Lemma \ref{lem.eigen}, the definition of a Margulis invariant given here is the same as the the definition of a Margulis invariant given in Definition 7.19 of \cite{Smilga4}. Abels--Margulis--Soifer \cite{AMS3} were the first to modify real valued Margulis invariants into vector valued invariants. Later, Smilga used similar invariants in \cite{Smilga,Smilga3,Smilga4} to construct proper affine actions of Schottky groups.
	\end{remark}
	\begin{proposition}\label{prop.marg}
		Let $g\in\sG$ be a loxodromic element, and let $h\in\sG$ be such that $g_\tth=h\exp(\tJd_g)h^{-1}$. Then, for any $Y\in\sV$, we have
		\[\chi^{red}_g(\tR_g)Y=\chi^{red}_g(1)\tR_h\tM(g,Y).\]
	\end{proposition}
	\begin{proof}
		Let $c:=h^{-1}gh$. Then by Proposition \ref{prop.dang}, we have $c\in\sM\sA$. Now, using Proposition \ref{prop.parpro}, we obtain that 
		$\chi^{red}_c(\tR_c)X=\chi^{red}_c(1)\pi_0(X)$ for all $X\in\sV$. Also, we have $\tP_c(x)=\tP_g(x)$. Hence, we deduce that
		\[\chi^{red}_g(\tR_g)Y=\tR_h\chi^{red}_c(\tR_c)\tR_h^{-1}Y=\chi^{red}_c(1)\tR_h\pi_0(\tR_h^{-1}Y)=\chi^{red}_g(1)\tR_h\tM(g,Y),\]
		and our result follows.
	\end{proof}
	
	\begin{notation}
		Suppose $\tR$ is a self-contragredient representation. We note that by Lemma 1.3 of \cite{Gros1} the representation $\tR$ admits an invariant symmetric bilinear form. We denote this bilinear form by $\tB_\tR$ and $\tB_\tR(v,v)$ by $\tB_\tR(v)$.
	\end{notation}
	
	\begin{corollary}\label{cor.margrat}
		Suppose $\tR$ is a self-contragredient representation, $g$ is loxodromic and $(g,X)\in\GV$. Then
		\[\tB_\tR(\tM(g,X))=\chi^{red}_g(1)^{-2}\tB_\tR(\chi^{red}_g(\tR_g)X)\] 
		is a rational expression in $(g,X)$.
	\end{corollary}
	\begin{proof}
		Suppose $g_\tth$ is the hyperbolic part of $g$ with respect to the Jordan decomposition and $h\in\sG$ is such that $g_\tth=h\exp(\tJd_g)h^{-1}$. We use Proposition \ref{prop.marg} and observe that 
		\[\tB_\tR(\tM(g,X))=\tB_\tR(\tR_h\tM(g,X))=\chi^{red}_g(1)^{-2}\tB_\tR(\chi^{red}_g(\tR_g)X).\]
		As $\tB_\tR$ is a symmetric bilinear form and $\chi^{red}_g$ is a polynomial, our result follows. 
	\end{proof}
	
	\section{Isospectral rigidity}\label{sec.nir}
	In this section, we prove various isospectral rigidity results related to Margulis invariant spectra of non trivial affine actions. Throughout this section we assume that $(\sG,\sV,\tR)$ satisfy Hypothesis \ref{hyp.1} and $\rho, \varrho:\Gamma\to\GV$ satisfy Hypothesis \ref{hyp.2} \footnote{The requirement that the linear part of the image of every non-identity element is loxodromic might be stronger than required for the arguments to follow. Similar arguments might also go through in the setting of $\sP$-proximal elements for some appropriate $\sP$ as discussed in Section 5 of \cite{Smilga4}.}.
	
	\subsection{Zero spectrum}
	
	We denote $\tL(\rho(\gamma))$ by $\tL_\rho(\gamma)$, $\tT(\rho(\gamma))$ by $\tT_\rho(\gamma)$ and $\tM(\rho(\gamma))$ by $\tM_\rho(\gamma)$.
	
	\begin{definition}
		The map $\tM_\rho:\Gamma\setminus\{e\}\to\sV$ (respectively $\tM_\varrho$) is called the \textit{marked} Margulis invariant spectrum of the representation $\rho$ (respectively $\varrho$).
	\end{definition}
	
	\begin{proposition}\label{prop.zd}
		Either $\rho(\Gamma)$ is Zariski dense inside $\GV$ or $\rho(\Gamma)$ is conjugate to $\tL_\rho(\Gamma)$ under the action of some element of $\{e\}\ltimes\sV$.
	\end{proposition}
	\begin{proof}
		Let $\sX$ be the Zariski closure of $\rho(\Gamma)$ inside $\GV$. We use Proposition 7.4.B.b of \cite{Hum} to deduce that $\tL(\sX)$ is Zariski closed inside $\sG$. We note that $\tL_\rho(\Gamma)\subset\tL(\sX)$ and $\tL_\rho(\Gamma)$ is Zariski dense in $\sG$. Hence, we deduce that $\tL(\sX)=\sG$.
		
		Now we consider the map $\left.\tL\right|_\sX: \sX\to\sG$ and note that \[\ker(\left.\tL\right|_\sX)=(\{e\}\ltimes\sV)\cap\sX.\]
		We will prove our result in two parts as follows:
		
		$\diamond$ If $\ker(\left.\tL\right|_\sX)$ is trivial, then $\left.\tL\right|_\sX$ is an isomorphism. Hence, for all $g\in\sG$ there exists $X_g\in\sV$ such that $X_{gh}=X_g+\tR_gX_h$ and $\sX=\{(g,X_g)\mid g\in\sG\}$. We use Whitehead's Lemma (see end of section 1.3.1 in page 13 of \cite{Raghu}) and deduce that there exists $X\in\sV$ such that $X_g=X-\tR_gX$. Therefore, we have $\tT_\rho(\gamma)=X-\tR_{\tL_\rho(\gamma)}X$ for all $\gamma\in\Gamma$. Hence $\rho(\gamma)=(e,X)(\tL_\rho(\gamma),0)(e,X)^{-1}$ for all $\gamma\in\Gamma$.
		
		$\diamond$ Let $\ker(\left.\tL\right|_\sX)$ be non trivial. We observe that $\ker(\left.\tL\right|_\sX)=\sX\cap\ker(\tL)$. As both $\sX$ and $\ker(\tL)=(\{e\}\ltimes\sV)$ are Zariski closed subgroups, we deduce that $\ker(\left.\tL\right|_\sX)$ is a Zariski closed subgroup. So it is a vector subspace of $\sV$. As $\tR$ is irreducible, we obtain that $\ker(\left.\tL\right|_\sX)=(\{e\}\ltimes\sV)$. Furthermore, as $\tL(\sX)=\sG$, we conclude that $\sX=\GV$.
		
		Therefore, either $\rho(\Gamma)$ is Zariski dense inside $\GV$ or $\rho(\Gamma)$ is conjugate to $\tL_\rho(\Gamma)$ under the action of some element of $\{e\}\ltimes\sV$ and our result follows.
	\end{proof}
	
	\begin{proposition}\label{prop.marg0}
		Suppose $\tM_\rho(\gamma)=0$ for all non identity $\gamma\in\Gamma$. Then $\rho(\Gamma)$ is conjugate to $\tL_\rho(\Gamma)$ under the action of some element of $\{e\}\ltimes\sV$. 
	\end{proposition}
	\begin{proof}
		As $\tM_\rho(\gamma)=0$ for all non identity $\gamma\in\Gamma$, using Proposition \ref{prop.marg}, we obtain that $\chi^{red}_{\tL_\rho(\gamma)}(\tR_{\tL_\rho(\gamma)})\tT_\rho(\gamma)=0$ for all non identity $\gamma\in\Gamma$. Also, by Remark \ref{rem.chired} we have $\chi^{red}_{\tL_\rho(e)}(\tR_{\tL_\rho(e)})\tT_\rho(e)=0$. We consider the following map:
		\begin{align*}
			f_0: \GV&\rightarrow\sV\\
			(g,X)&\mapsto \chi^{red}_g(\tR_g)X
		\end{align*}
		and observe that it is algebraic. We denote the zero set of $f_0$ by $\sZ(f_0)$ i.e.
		\[\sZ(f_0):=\{(g,X)\in\GV\mid f_0(g,X)=0\}.\]
		We choose $X\neq0$ inside $\sV^0$ and a loxodromic element $g\in\sG$ such that the dimension of the unit eigenspace of $\tR_g$ is exactly $\dim\sV^0$. Moreover, let $h\in\sG$ be such that $hgh^{-1}\in\sM\sA$. Then using Lemma \ref{lem.reg} and Proposition \ref{prop.marg} we obtain that
		\[f_0(g,\tR_hX)=\chi^{red}_g(\tR_g)\tR_hX=\chi^{red}_g(1)\tR_h\pi_0(X)=\chi^{red}_g(1)\tR_hX\neq0.\]
		Hence $\sZ(f_0)\subsetneq\GV$ and it follows that $\sX$, the Zariski closure of $\rho(\Gamma)$ inside $\GV$, is a proper subvariety of $\GV$ i.e. $\sX\subset\sZ(f_0)\subsetneq\GV$. Finally, we conclude our result by using Proposition \ref{prop.zd}.
	\end{proof}
	
	\begin{corollary}\label{cor.marlin}[refer to Theorem \ref{thm.marlin}]
		Let $\tL_\rho=\tL_\varrho$ and let $\tM_\rho(\gamma)=\tM_\varrho(\gamma)$ for all non identity $\gamma\in\Gamma$. Then there exists an inner automorphism $\sigma$ of $\GV$ such that $\sigma\circ\rho=\varrho$.
	\end{corollary}
	\begin{proof}
		Let $\eta:=(\tL_\rho,\tT_\rho-\tT_\varrho)$. We observe that for all non identity $\gamma\in\Gamma$, we have
		\[\tM_\eta(\gamma)=\tM_\rho(\gamma)-\tM_\varrho(\gamma)=0.\]
		Therefore, using Proposition \ref{prop.marg0} we obtain that there exists $Y\in\sV$ such that $\eta(\gamma)=(e,Y)(\tL_\rho(\gamma),0)(e,Y)^{-1}$
		for all $\gamma\in\Gamma$. Hence, for all $\gamma\in\Gamma$ it follows that $\tT_\rho(\gamma)-\tT_\varrho(\gamma)=Y-\tR_{\tL_\rho(\gamma)}Y$ and we conclude by observing that for all $\gamma\in\Gamma$, the following holds: $\rho(\gamma)=(e,Y)\varrho(\gamma)(e,Y)^{-1}$.
	\end{proof}
	
	\begin{corollary}\label{cor.JM}[refer to Theorem \ref{thm.isoJM}]
		Suppose $\mathsf{rank}(\sG)\geq2$ and $(\tJd,\tM)(\rho(\gamma))=(\tJd,\tM)(\varrho(\gamma))$ for all non identity $\gamma\in\Gamma$. Then there exists an inner automorphism $\sigma$ of $\GV$ such that $\sigma\circ\rho=\varrho$.
	\end{corollary}
	\begin{proof}
		As $\tJd(\rho(\gamma))=\tJd(\varrho(\gamma))$ for all $\gamma\in\Gamma$, using Theorem B of Dal'bo--Kim \cite{DK} (see also the Criterion in \cite{DK2}) we obtain that $\tL_\rho=g\tL_\varrho g^{-1}$ for some $g\in\sG$. We observe from the definition that Margulis invariants are invariant under conjugation. It follows that \[\tM(g\tL_\varrho(\gamma)g^{-1},\tR_g\tT_\varrho(\gamma))=\tM(\tL_\varrho(\gamma),\tT_\varrho(\gamma)).\]
		Hence, we deduce that
		\[\tM(\tL_\rho(\gamma),\tR_g\tT_\varrho(\gamma))=\tM(\tL_\rho(\gamma),\tT_\rho(\gamma)).\]
		Therefore, $\tM(\tL_\rho(\gamma),\tT_\rho(\gamma)-\tR_g\tT_\varrho(\gamma))=0$ for all non identity $\gamma\in\Gamma$. Finally, we observe that $(\tL_\rho,\tT_\rho-\tR_g\tT_\varrho):\Gamma\to\GV$ is a homomorphism and use Corollary \ref{cor.marlin} to conclude that there exists an inner automorphism $\sigma$ of $\GV$ such that $\sigma\circ\rho=\varrho$.
	\end{proof}
	
	\begin{proposition}\label{prop.margnorm0}
		Suppose $\tR$ is a self-contragredient representation and $\tB_\tR(\tM_\rho(\gamma))=0$ for all non identity $\gamma\in\Gamma$. Then $\rho(\Gamma)$ is conjugate to $\tL_\rho(\Gamma)$ under the action of some element of $\{e\}\ltimes\sV$. 
	\end{proposition}
	\begin{proof}
		As $\tB_\tR(\tM_\rho(\gamma))=0$ for all non identity $\gamma\in\Gamma$, using Proposition \ref{prop.marg}, we obtain that $\tB_\tR(\chi^{red}_{\tL_\rho(\gamma)}(\tR_{\tL_\rho(\gamma)})\tT_\rho(\gamma))=0$ for all non identity $\gamma\in\Gamma$. Also, by Remark \ref{rem.chired}, we have $\tB_\tR(\chi^{red}_{\tL_\rho(e)}(\tR_{\tL_\rho(e)})\tT_\rho(e))=0$. We consider the following map:
		\begin{align*}
			\tB(f_0): \GV&\rightarrow\R\\
			(g,X)&\mapsto \tB_\tR(\chi^{red}_g(\tR_g)X)
		\end{align*}
		and observe that it is algebraic. We denote the zero set of $\tB(f_0)$ by $\sZ(\tB(f_0))$ i.e.
		\[\sZ(\tB(f_0)):=\{(g,X)\in\GV\mid\tB(f_0)(g,X)=0\}.\]
		We use Lemma \ref{lem.ue} and Lemma \ref{lem.reg} to obtain that there exists $c\in\sM\sA$ such that $\chi^{red}_c(1)\neq0$. As $\tB_\tR$, the invariant form of $\tR$, is a symmetric bilinear form, we use Lemma 1.1 of \cite{Gros1} to conclude that the restriction of $\tB_\tR$ on $\sV^0$ is a non-degenerate symmetric bilinear form. Hence, $\sV^0$ admits vectors which are not self-orthogonal. Let $X\in\sV^0$ be such that $\tB_\tR(X)\neq0$. We use Proposition \ref{prop.parpro} and obtain
		\[\tB(f_0)(c,X)=\tB_\tR(\chi^{red}_c(\tR_c)X)=\chi^{red}_c(1)^2\tB_\tR(\pi_0(X))=\chi^{red}_c(1)^2\tB_\tR(X)\neq0.\]
		
		Hence $\sZ(\tB(f_0))\subsetneq\GV$ and it follows that $\sX$, the Zariski closure of $\rho(\Gamma)$ inside $\GV$, is a proper subvariety of $\GV$ i.e. $\sX\subset\sZ(\tB(f_0))\subsetneq\GV$. Finally, we conclude our result by using Proposition \ref{prop.zd}.
	\end{proof}
	
	\begin{corollary}\label{cor.tfae}
		The following are equivalent:
		\begin{enumerate}
			\item $\tM_\rho(\gamma)=0$ for all non identity $\gamma\in\Gamma$,
			\item $\rho(\Gamma)$ is conjugate to $\tL_\rho(\Gamma)$ under the action of some element of $\{e\}\ltimes\sV$.
		\end{enumerate}
		Moreover, if $\tR$ is a self-contragredient representation. Then both are equivalent to:
		\begin{enumerate}
			\item[3.] $\tB_\tR(\tM_\rho(\gamma))=0$ for all non identity $\gamma\in\Gamma$.
		\end{enumerate}
	\end{corollary}
	\begin{proof}
		$(1)\implies(2)$ is Proposition \ref{prop.marg0}, $(3)\implies(2)$ is Proposition \ref{prop.margnorm0} and $(2)\implies(1)\implies(3)$ is straightforward.
	\end{proof}
	
	\begin{corollary}\label{cor.zd}
		Suppose there exists a $\gamma\in\Gamma$ such that $\tM_\rho(\gamma)\neq0$. Then, $\rho(\Gamma)$ is Zariski dense inside $\GV$. 
	\end{corollary}
	\begin{proof}
		We use Proposition \ref{prop.zd} to obtain that either $\rho(\Gamma)$ is Zariski dense inside $\GV$ or $\rho(\Gamma)$ is conjugate to $\tL_\rho(\Gamma)$ under the action of some element of $\{e\}\ltimes\sV$. We observe that if $\rho(\Gamma)$ is not Zariski dense inside $\GV$, then $\rho(\Gamma)$ is conjugate to $\tL_\rho(\Gamma)$ under the action of some element of $\{e\}\ltimes\sV$ and we obtain a contradiction using Corollary \ref{cor.tfae}.
	\end{proof}
	
	\subsection{General case}
	In this subsection we prove isopectrality results under the assumption that $\tR$ admit invariant norms.
	
	\begin{lemma}\label{lem.alg}
		Suppose $\tR$ is a self-contragredient representation. Also, suppose $(g,X),(h,Y)\in\GV$ are such that their linear parts are loxodromic and $\tB_\tR(\tM(g,X))=\tB_\tR(\tM(h,Y))$. Then
		\[\tB_\tR(\chi^{red}_g(1)\chi^{red}_h(\tR_h)Y)=\tB_\tR(\chi^{red}_h(1)\chi^{red}_g(\tR_g)X).\]
	\end{lemma}
	\begin{proof}
		We use Proposition \ref{prop.marg} and observe that
		\begin{align*}
			\tB_\tR(\chi^{red}_g(1)\chi^{red}_h(\tR_h)Y)&=\tB_\tR(\chi^{red}_g(1)\chi^{red}_h(1)\tM(h,Y))\\
			&=\tB_\tR(\chi^{red}_g(1)\chi^{red}_h(1)\tM(g,X))=\tB_\tR(\chi^{red}_h(1)\chi^{red}_g(\tR_g)X).
		\end{align*}
		Our result follows.
	\end{proof}
	
	\begin{notation}\label{not.iota}
		Let $\sN_r,\sN_l$ be two nontrivial proper normal subgroups of $\GV$ for which there exists a continuous isomorphism $\iota:(\GV)/\sN_r\to(\GV)/\sN_l$. We denote the set of all $(g_\iota,X_{g_\iota}, g_\iota^\prime, Y_{g_\iota})\in\left(\GV\times\GV\right)$ such that $(g_\iota^\prime,Y_{g_\iota})\sN_l=\iota((g_\iota, X_{g_\iota})\sN_r)$ by $\sD_\iota$ i.e.
		\[\sD_\iota:=\{(g_\iota,X_{g_\iota}, g_\iota^\prime, Y_{g_\iota})\mid(g_\iota^\prime,Y_{g_\iota})\sN_l=\iota((g_\iota, X_{g_\iota})\sN_r) \}.\]
	\end{notation}
	
	\begin{lemma}\label{lem.neq}
		Suppose $\tR$ is a self-contragredient representation and ${f}:\GV\times\GV\to\R$ is such that for all $(g,X,h,Y)\in\GV\times\GV$, we have:
		\[{f}(g,X,h,Y):=\tB_\tR(\chi^{red}_g(1)\chi^{red}_h(\tR_h)Y)-\tB_\tR(\chi^{red}_h(1)\chi^{red}_g(\tR_g)X),\]
		and let $\sZ({f}):=\{(g,X,h,Y)\in\GV\times\GV\mid {f}(g,X,h,Y)=0\}$.
		Then $\sD_\iota\not\subset\sZ({f})$ for all $\iota$ as mentioned in Notation \ref{not.iota}. In particular, we have $\sZ({f})\subsetneq\GV\times\GV$. 
	\end{lemma}
	\begin{proof}
		Let $\sN_r,\sN_l$ be any two nontrivial proper normal subgroups of $\GV$ such that $\iota:(\GV)/\sN_r\to(\GV)/\sN_l$ is a continuous isomorphism. We use Proposition \ref{prop.normal} and observe that $\sN_r=\sG_r\ltimes\sV$ and $\sN_l=\sG_l\ltimes\sV$, for some proper normal subgroup $\sG_r,\sG_l$ of $\sG$. Now using the third isomorphism Theorem of groups we obtain that $(\GV)/\sN_r$ is isomorphic to $\sG/\sG_r$ and $(\GV)/\sN_l$ is isomorphic to $\sG/\sG_l$. Therefore, $\iota:(\GV)/\sN_r\to(\GV)/\sN_l$ gives rise to an isomorphism $\iota:\sG/\sG_r\to\sG/\sG_l$. Now we use Lemmas \ref{lem.ue} and \ref{lem.reg} to observe that the set $S:=\{g\in\sG\mid\tP_g(0)\neq0\}$ is an open dense subset of $\sG$. Moreover, as $\sG/\sG_r$ and $\sG/\sG_l$ are the quotients of $\sG$ by some group action, the projection maps $\pi_r:\sG\to\sG/\sG_r$ and $\pi_l:\sG\to\sG/\sG_l$ are open. Hence $\pi_r(S)$ and $\pi_l(S)$ are open dense subsets of $\sG/\sG_r$ and $\sG/\sG_l$ respectively. It follows that $\iota\circ\pi_r(S)$ is an open dense subset of $\sG/\sG_l$ and hence $\iota\circ\pi_r(S)\cap\pi_l(S)$ is an open dense subset of $\sG/\sG_l$. Let $p\in\iota\circ\pi_r(S)\cap\pi_l(S)$. Then, there exist $g_\iota,g_\iota^\prime\in S$ such that $p=\pi_l(g_\iota^\prime)=\iota\circ\pi_r(g_\iota)$ i.e. $g_\iota^\prime\sG_l=\iota(g_\iota\sG_r)$. It follows that $\tP_{g_\iota}(0)\neq0$, $\tP_{g_\iota^\prime}(0)\neq0$ and $(g_\iota^\prime,Y)\sN_l=\iota((g_\iota,X)\sN_r)$ for all $X,Y\in\sV$.
		
		As $\tB_\tR$, the invariant form of $\tR$, is a symmetric bilinear form, we use Lemma 1.1 of \cite{Gros1} to conclude that the restriction of $\tB_\tR$ on $\sV^0$ is a non-degenerate symmetric bilinear form. Hence, $\sV^0$ admits vectors which are not self-orthogonal. Let $V\in\sV^0$ be such that $\tB_\tR(V)\neq0$. Moreover, let $h\in\sG$ be such that $hg_\iota h^{-1}\in\sM\sA$. We choose $Y_{g_\iota}=0$, $X_{g_\iota}=\tR_hV$ and using Proposition \ref{prop.parpro} we observe that
		\begin{align*}
			{f}(g_\iota,X_{g_\iota},g_\iota^\prime,Y_{g_\iota})&={f}(g_\iota,X_{g_\iota},g_\iota^\prime,0)\\
			&=\tB_\tR(\chi^{red}_{g_\iota}(1)\chi^{red}_{g_\iota^\prime}(\tR_{g_\iota^\prime})0)-\tB_\tR(\chi^{red}_{g_\iota^\prime}(1)\chi^{red}_{g_\iota}(\tR_{g_\iota})X_{g_\iota})\\
			&=-\tB_\tR(\chi^{red}_{g_\iota^\prime}(1)\chi^{red}_{g_\iota}(1)\tR_h\circ\pi_0\circ\tR_h^{-1}(X_{g_\iota}))\\
			&=-\tB_\tR(\chi^{red}_{g_\iota^\prime}(1)\chi^{red}_{g_\iota}(1)V)\neq0.
		\end{align*}
		Hence, the set $\left(\GV\times\GV\right)\setminus\sZ({f})$ is non empty and in particular it contains $(g_\iota,X_{g_\iota},g_\iota^\prime, Y_{g_\iota})$ with $(g_\iota^\prime,Y_{g_\iota})\sN_l=\iota((g_\iota,X_{g_\iota})\sN_r)$, concluding our result.
	\end{proof}
	
	\begin{proposition}\label{prop.affauto}
		Suppose $\rho(\Gamma)$ and $\varrho(\Gamma)$ are both Zariski dense inside $\GV$. Suppose $\sX$, the Zariski closure of $(\varrho,\rho)(\Gamma)$ inside $\GV\times\GV$, is such that $\sD_\iota\not\subset\sX$ for all $\iota$ mentioned in Notation \ref{not.iota}. Then there exists a continuous automorphism $\sigma:\GV\to\GV$ such that $\sigma\circ\rho=\varrho$.
	\end{proposition}
	\begin{proof}
		As $(\varrho,\rho)(\gamma)(\varrho,\rho)(\Gamma)(\varrho,\rho)(\gamma)^{-1}=(\varrho,\rho)(\Gamma)$ for all $\gamma\in\Gamma$, we obtain that $(\varrho,\rho)(\gamma)\sX(\varrho,\rho)(\gamma)^{-1}=\sX$. We denote the projections onto the left and right components of $\GV\times\GV$ by $\pi_l$ and $\pi_r$ respectively, i.e. $\pi_l,\pi_r: \GV\times\GV\to\GV$ be such that for all $(g,X,h,Y)\in\GV\times\GV$ we have $\pi_l(g,X,h,Y)=(g,X)$ and $\pi_r(g,X,h,Y)=(h,Y)$.
		We observe that $\pi_l$ and $\pi_r$ are homomorphisms. We use Proposition 7.4.B.b of \cite{Hum} to deduce that both $\pi_l(\sX)$ and $\pi_r(\sX)$ are Zariski closed subgroups of $\GV$. We observe that $\pi_l(\sX)\supset\varrho(\Gamma)$ and $\pi_r(\sX)\supset\rho(\Gamma)$. As both $\rho(\Gamma)$ and $\varrho(\Gamma)$ are Zariski dense inside $\GV$, we conclude that $\pi_l(\sX)=\GV=\pi_r(\sX)$.
		
		Now we consider the following two normal subgroups of $\sX$: $\sN_l:=\ker(\left.\pi_l\right|_\sX)$ and $\sN_r:=\ker(\left.\pi_r\right|_\sX)$. As $(\varrho,\rho)(\Gamma)\subset\sX$ and $\sN_l$ is normal in $\sX$, for all $\gamma\in\Gamma$ we have $(\varrho(\gamma),\rho(\gamma))\sN_l(\varrho(\gamma),\rho(\gamma))^{-1}\subset \sN_l$. Moreover, as $\sN_l=\ker(\pi_l)\cap\sX$, we obtain that $\sN_l\subset\{o\}\times\GV$, where $o:=(e,0)$. Hence any element of $\sN_l$ is of the form $(o,n)$ and we obtain that
		\[(\varrho(\gamma),\rho(\gamma))(o,n)(\varrho(\gamma),\rho(\gamma))^{-1}=(o,\rho(\gamma))(o,n)(o,\rho(\gamma))^{-1}\]
		for all $\gamma\in\Gamma$. As $\rho(\Gamma)$ is Zariski dense inside $\GV$, we obtain that $\sN_l$ is normal inside $\{o\}\times\GV$. Similarly, we obtain that $\sN_r$ is normal inside $\GV\times\{o\}$. Moreover, as $\pi_l(\sX)=\GV=\pi_r(\sX)$, we obtain that
		\[\dim(\sN_l)=\dim(\sX)-\dim(\GV)=\dim(\sN_r).\]
		Now using Proposition \ref{prop.normal} we deduce that one of the following holds:
		\begin{enumerate}
			\item $\sN_l=\{o\}\times\GV$ and $\sN_r=\GV\times\{o\}$,
			\item $\sN_l=\{o\}\times\sG_l\ltimes\sV$ and $\sN_r=\sG_r\ltimes\sV\times\{o\}$ for some nontrivial proper normal subgroups $\sG_l,\sG_r$ of $\sG$,
			\item both are trivial.
		\end{enumerate}
		We consider these three cases separately below:
		
		$\diamond$ If $\sN_l=\{o\}\times\GV$ and $\sN_r=\GV\times\{o\}$, then we obtain a contradiction. Indeed, we have $\GV\times\GV=\sN_r\sN_l\subset\sX\subsetneq\GV\times\GV$.
		
		$\diamond$ Suppose $\sN_l=\{o\}\times\sG_l\ltimes\sV$ and $\sN_r=\sG_r\ltimes\sV\times\{o\}$. Then, by Goursat's lemma \cite{Gour}, we get that the image of $\sX$ inside $(\GV)/\sN_r\times(\GV)/\sN_l$ is given by the graph of an isomorphism $\sigma:(\GV)/\sN_r\to(\GV)/\sN_l$. 
		
		Now we want to show that $\sigma$ is continuous. We consider the projections
		\begin{align*}
			p_r&:\GV\times\{o\}\to(\GV\times\{o\})/(\sG_r\ltimes\sV\times\{o\}),\\
			p_l&:\{o\}\times\GV\to(\{o\}\times\GV)/(\{o\}\times\sG_l\ltimes\sV).
		\end{align*}
		Let $\pi_l^\prime: \sX/(\sN_r\sN_l)\to(\GV)/\sN_r$ and $\pi_r^\prime: \sX/(\sN_r\sN_l)\to(\GV)/\sN_l$ respectively be the quotient maps induced by $p_r\circ(\left.\pi_l\right|_\sX)$ and $p_l\circ(\left.\pi_r\right|_\sX)$. We note that $\sigma=(\pi_l^\prime)^{-1}\circ\pi_r^\prime$. Hence $\sigma$ is a continuous isomorphism. It follows that, for all $g,g^\prime\in\sG$, and $X,X^\prime\in\sV$ with $\sigma((g,X)\sN_r)=(g^\prime,X^\prime)\sN_l$ we have $(g,X,g^\prime,X^\prime)\in\sX$ i.e. $\sD_\sigma\subset\sX$. Hence, $\emptyset=\sD_\sigma\setminus\sX\neq\emptyset$, a contradiction.
		
		$\diamond$ Suppose both $\sN_l$ and $\sN_r$ are trivial. Then, by Goursat's lemma \cite{Gour}, we get that $\sX$ inside $\GV\times\GV$ is the graph of an automorphism $\sigma$ of $\GV$. We can choose $\sigma$ to be $(\left.\pi_l\right|_\sX)^{-1}\circ(\left.\pi_r\right|_\sX)$ and conclude by observing that it is continuous.
	\end{proof}
	
	\begin{proposition}\label{prop.contaut}
		Let $\sigma:\GV\to\GV$ be a continuous automorphism. Then there exists $(\tA,Y)\in\mathsf{GL}(\sV)\ltimes\sV$ such that $\tA$ normalizes $\tR(\sG)$ and the action of $\sigma$ is conjugation by $(\tA,Y)$.
	\end{proposition}
	\begin{proof}
		We observe that $\sigma$ induces a continuous additive map $\tilde{\sigma}:\sV\to\sV$. As continuous additive maps between vector spaces are linear and $\sigma$ is an automorphism, $\tilde{\sigma}$ is an invertible linear map. Hence, there exists $\tA\in\mathsf{GL}(\sV)$ such that $\sigma(e,X)=(e,\tA X)$ for all $X\in\sV$. Moreover, for $g\in\sG$, let $g_{\sigma}\in\sG$ and $Y_{g_\sigma}\in\sV$ be such that $\sigma(g,0)=(g_\sigma,Y_{g_\sigma})$. Then $Y_{g_\sigma h_\sigma}=Y_{g_\sigma}+\tR_{g_\sigma}Y_{h_\sigma}$ for all $g_{\sigma},h_{\sigma}\in\sG$. We use Whitehead's Lemma (see end of section 1.3.1 in page 13 of \cite{Raghu}) to deduce that there exists $Y\in\sV$ such that $Y_{g_\sigma}=Y-\tR_{g_\sigma}Y$. We also note that for all $g\in\sG$ we have $\tA\tR_g=\tR_{g_\sigma}\tA$. Indeed, for any $X\in\sV$:
		\[(g_\sigma,Y_{g_\sigma}+\tA X)=\sigma(e,X)\sigma(g,0)=\sigma(g,0)\sigma(e,\tR_g^{-1}X)=(g_\sigma,Y_{g_\sigma}+\tR_{g_\sigma}\tA\tR_g^{-1}X),\]
		and it follows that $\sigma(g,X)=(\tA,Y)(\tR_g,X)(\tA,Y)^{-1}$.
	\end{proof}
	
	\begin{theorem}\label{thm.normain}(refer to Theorem \ref{thm.normainintro})
		Suppose $\tR$ is a self-contragredient representation and $\tB_\tR(\tM_\rho(\gamma))=\tB_\tR(\tM_\varrho(\gamma))$ for all non identity $\gamma\in\Gamma$. Then either of the following holds:
		\begin{enumerate}
			\item there exist $X,Y\in\sV$ such that $\rho$ is conjugate to $\tL_\rho$ by $(e,X)$ and $\varrho$ is conjugate to $\tL_\varrho$ by $(e,Y)$,
			\item there exists $(\tA,Y)\in\mathsf{GL}(\sV)\ltimes\sV$ such that $\tA$ normalizes $\tR(\sG)$ and $\rho$ is conjugate to $\varrho$ by $(\tA,Y)$.
		\end{enumerate}
		Moreover, if $\tL_\rho=\tL_\varrho$, then $\tA$ centralizes $\tR(\sG)$.
	\end{theorem}
	\begin{proof}
		We will prove this result in three parts.
		
		$\diamond$ Suppose $\tB_\tR(\tM_\rho(\gamma))=0$ for all non identity $\gamma\in\Gamma$. It follows that $\tB_\tR(\tM_\varrho(\gamma))=0$ for all non identity $\gamma\in\Gamma$. We use Corollary \ref{cor.tfae} and obtain that there exists $X,Y\in\sV$ such that $\rho=(e,X)\tL_\rho(e,X)^{-1}$ and $\varrho=(e,Y)\tL_\varrho(e,Y)^{-1}$. It follows that both $\rho(\Gamma)$ and $\varrho(\Gamma)$ are Zariski dense inside some conjugates of $\sG$.
		
		$\diamond$ Suppose there exists $\gamma\in\Gamma$ such that $\tB_\tR(\tM_\rho(\gamma))\neq0$. Hence $\tB_\tR(\tM_\varrho(\gamma))\neq0$, and using Corollary \ref{cor.zd}, we obtain that both $\rho(\Gamma)$ and $\varrho(\Gamma)$ are Zariski dense inside $\GV$.
		
		We use Lemma \ref{lem.alg} and Remark \ref{rem.chired} to observe that ${f}(\varrho(\gamma),\rho(\gamma))=0$ for all $\gamma\in\Gamma$. Hence $\sX$, the Zariski closure of $(\varrho,\rho)(\Gamma)$ inside $\GV\times\GV$, is a subvariety of $\sZ({f})$ and using Lemma \ref{lem.neq} we obtain that $\sD_\iota\not\subset\sX$ for all $\iota$ as mentioned in Notation \ref{not.iota}. Now, using Proposition \ref{prop.affauto}, we deduce that there exists a continuous automorphism $\sigma:\GV\to\GV$ such that $\sigma\circ\rho=\varrho$. Finally, using Proposition \ref{prop.contaut}, we conclude that there exists $(\tA,Y)\in\mathsf{GL}(\sV)\ltimes\sV$ such that $\tA$ normalizes $\tR(\sG)$ and $\sigma(g,X)=(\tA,Y)(\tR_g,X)(\tA,Y)^{-1}$.
		
		$\diamond$ Finally, if $\tL_\rho=\tL_\varrho$, then, using $\sigma(g,X)=(\tA,Y)(\tR_g,X)(\tA,Y)^{-1}$, we obtain that $\tA\tR_g=\tR_g\tA$ for all $g\in\sG$. Hence, $\tA$ is in the centralizer of $\tR(\sG)$ and our result follows.
	\end{proof}
	
	\begin{remark}
		We note that, by Schur's lemma, $\tA$ centralizes $\tR(\sG)$ means in fact, that $\tA$ is scalar (and so is equal to $\pm I$).
	\end{remark}
	
	\begin{corollary}\label{cor.maintech}(refer to Theorem \ref{thm.normainintro})
		Let $\tR$ be a self-contragredient representation and let $\tM_\rho(\gamma)=\tM_\varrho(\gamma)$ for all non identity $\gamma\in\Gamma$. Then either of the following holds:
		\begin{enumerate}
			\item both $\rho(\Gamma)$ and $\varrho(\Gamma)$ are Zariski dense inside some conjugates of $\sG$,
			\item there exists a continuous automorphism $\sigma:\GV\to\GV$ such that $\sigma\circ\rho=\varrho$ and $\sigma$ is conjugation by an element $(\tA,Y)\in\mathsf{GL}(\sV)\ltimes\sV$ such that $\tA$ normalizes $\tR(\sG)$.
		\end{enumerate}
	\end{corollary}
	\begin{proof}
		As $\tM_\rho(\gamma)=\tM_\varrho(\gamma)$ for all non identity $\gamma\in\Gamma$, we obtain $\tB_\tR(\tM_\rho(\gamma))=\tB_\tR(\tM_\varrho(\gamma))$ for all non identity $\gamma\in\Gamma$. Hence, using Theorem \ref{thm.normain} we obtain our result.
	\end{proof}

	\section{Applications}\label{sec.appli}
	In this section, we define Margulis--Smilga manifolds, and prove that the Margulis invariant spectra of certain special Margulis--Smilga manifolds determine them uniquely up to conjugacy. Throughout this section we assume that $(\sG,\sV,\tR)$ satisfy Hypothesis \ref{hyp.1} and $\rho, \varrho:\Gamma\to\sG\ltimes_\tR\sV$ satisfy Hypothesis \ref{hyp.2}.
	
	\begin{definition} Let $\Gamma$ be a finitely generated non-abelian free group, $\sG$ be a real split semisimple algebraic group with trivial center and $\rho:\Gamma\to\GV$ be a loxodromic representation whose linear part $\tL(\rho(\Gamma))$ is Zariski dense inside $\sG$. 
		\begin{enumerate}
			\item If $\rho(\Gamma)$ acts properly discontinuously and freely on $\sV$, then $\rho(\Gamma)\backslash\sV$ is called a \emph{Margulis--Smilga manifold}.
			\item Moreover, if $\tR$ is self-contragredient \footnote{If $\rho(\Gamma)\backslash\sV$ is a Margulis--Smilga manifold, then by Kostant--Sullivan \cite{KS}, every element of $\tR(\sG)$ has $1$ as an eigenvalue. As $\tR$ is self-contragredient, it follows that, $\tR$ preserves a non-degenerate symmetric bilinear form of signature $(p,q)$ with $pq\neq0$.}, then a Margulis--Smilga manifold is called a \emph{Margulis--Smilga spacetime}.
			\item In particular, when $\sG=\mathsf{SO}(2n,2n-1)$, $\sV=\R^{4n-1}$ and $\tR$ is the linear action, we call a Margulis--Smilga spacetime as a \emph{Margulis spacetime}.
		\end{enumerate}
	\end{definition}
	
	We note that the study of proper actions of discrete subgroups of affine group goes back to more than a century. It started with Bieberbach \cite{B1,B2}, who showed that any discrete subgroup of $\mathsf{O}(n,\R)\ltimes\R^n$ that acts properly discontinously, freely, and cocompactly on $\R^n$ is virtually a product of infinite cyclic groups. Later, Auslander--Marcus \cite{AM} constructed examples of discrete subgroups of affine transformations, $\mathsf{GL}(n,\R)\ltimes\R^n$ that act properly discontinously, freely, and cocompactly on $\R^n$ but are not virtually a product of infinite cyclic groups. These examples are virtually polycyclic. Later, in a failed attempt Auslander tried to show that any discrete affine group acting properly discontinuously, freely and cocompactly on $\R^n$ is virtually a polycyclic group. The statement was later renamed as the Auslander's conjecture. The conjecture is still open and is known only for dimensions less than seven by works of Fried--Goldman \cite{FG}, Abels--Margulis--Soifer \cite{AMS2} and Tomanov \cite{Tom}. Milnor \cite{Milnor} wondered about the possibility of dropping the assumption of cocompact action from the conjecture. In response to Milnor, Margulis \cite{Margulis1,Margulis2} produced examples of non abelian free subgroups of $\mathsf{SO}(2,1)\ltimes\R^3$ that act properly discontinuously and freely on $\R^3$. The construction of Margulis was later generalized by Abels--Margulis--Soifer \cite{AMS} and Smilga \cite{Smilga, Smilga3, Smilga4} (for a more detailed survey see \cite{DDGS}). More examples of such actions by Coxeter groups, using different techniques, were obtained by Danciger--Gu\' eritaud--Kassel \cite{DGK3}. 
	
	A key object in the detection of proper affine actions is the Margulis invariant. This was first introduced by Margulis and was later generalized in the works of Abels--Margulis--Soifer \cite{AMS,AMS3} and Smilga \cite{Smilga,Smilga3,Smilga4}. The properties of the marked Margulis invariant spectrum of a discrete subgroup of $\GV$ determine whether the action of the subgroup on $\sV$ is proper or not (see the works of Goldman--Labourie--Margulis \cite{GLM}, Danciger--Gu\' eritaud--Kassel \cite{DGK1}, Ghosh--Treib \cite{GT} and Ghosh \cite{Ghosh7} for more details). Later, in the works of Drumm--Goldman \cite{DG}, Charette--Drumm \cite{CD} and Kim \cite{Kim}, isospectral rigidity results for Margulis spacetimes were obtained (see also \cite{Ghosh4}). As loxodromic representations which give rise to Margulis--Smilga manifolds do not admit global fixed points, we use the results proved in the previous sections to generalize these results for Margulis--Smilga spacetimes and Margulis--Smilga manifolds:
	
	\begin{corollary} \label{cor.main}
		Let $\rho$ and $\varrho$ be two Margulis--Smilga manifolds. Then the following hold:
		\begin{enumerate}
			\item If $\rho$ and $\varrho$ are conjugate via some inner automorphism of $\GV$, then they have the same Jordan--Margulis invariant spectrum.
			\item If $\rho$, $\varrho$ have the same Margulis invariant spectrum and $\tL_\rho=\tL_\varrho$, then there exists $\sigma$, an inner automorphism of $\GV$, such that $\rho=\sigma\circ\varrho$.
			\item If $\rho$, $\varrho$ have the same Jordan--Margulis invariant spectrum and also $\mathsf{rank}(\sG)\geq2$, then there exists $\sigma$, an inner automorphism of $\GV$, such that $\rho=\sigma\circ\varrho$.
			\item If $\rho$, $\varrho$ have the same Margulis invariant spectrum and $\tR$ is a self-contragredient representation, then there exists $(\tA,Y)\in\mathsf{GL}(\sV)\ltimes\sV$ such that $\rho=(\tA,Y)\varrho(\tA,Y)^{-1}$. 
		\end{enumerate}
	\end{corollary}
	\begin{proof} 
		We observe that the first point follows from the definition of Jordan--Margulis invariants, the second point follows from Corollary \ref{cor.marlin}, and the third point follows from Corollary \ref{cor.JM}.
		
		Finally, for the fourth point, suppose $\rho$, $\varrho$ are two Margulis--Smilga manifolds with the same Margulis invariant spectrum, and $\tR$ is a self contragredient representation. Hence $\rho$ and $\varrho$, act properly on $\sV$. It follows that for some $\gamma\in\Gamma\setminus\{e\}$, we have $\tM_\rho(\gamma)\neq0$ and $\tM_\varrho(\gamma)\neq0$. Now we use Corollaries \ref{cor.zd} and \ref{cor.maintech} to obtain our result.
	\end{proof}
	
	\begin{corollary}
		Let $\sG$ be a real split simple algebraic Lie group with trivial center, let $\fg$ be its Lie algebra with Killing form $\tB$, let $\tAd:\sG\to\mathsf{GL}(\fg)$ be the adjoint representation, and let $\rho$ and $\varrho$ be two Margulis--Smilga spacetimes. Then the following hold:
		\begin{enumerate}
			\item If $\rho$ and $\varrho$ are conjugate via some inner automorphism of $\sfG$, then they have the same Jordan--Margulis invariant spectrum.
			\item If $\rho$, $\varrho$ have the same Margulis invariant spectrum and $\tL_\rho=\tL_\varrho$, then there exists $\sigma$, an inner automorphism of $\sfG$, such that $\rho=\sigma\circ\varrho$.
			\item If $\rho$, $\varrho$ have the same Margulis invariant spectrum then there exists $(\tA,Y)\in\mathsf{GL}(\fg)\ltimes\fg$ such that $\rho=(\tA,Y)\varrho(\tA,Y)^{-1}$. 
		\end{enumerate}
	\end{corollary}
	\begin{proof}
		We observe that $\fg$ is finite dimensional, and the Killing form $\tB$ is a non-degenerate symmetric bilinear form. As $\sG$ is Zariski-connected and has a trivial center, we obtain that the adjoint representation is faithful. Also, as $\sG$ is simple we obtain that $\tAd$ is irreducible. It follows that the complexification of $\tAd$ is also irreducible, and hence, $\tAd$ is absolutely irreducible. Moreover, by Proposition 4.4.5 of \cite{Sp} we obtain that $\tAd$ is algebraic. We observe that $\fa$ is a zero weight space of $\tAd$ and hence $\tAd$ admits zero as a weight. Also, as the Killing form is not degenerate, we obtain that $\tAd$ is self-contragredient. Hence our result follows from Corollary \ref{cor.main}.
	\end{proof}

	\appendix
	\section{Normal subgroups}
	
	In this section, we prove some results about the normal subgroups of affine groups of the form $\GV$, where $\sG$ is a real split semisimple algebraic Lie group with trivial center acting on a vector space $\sV$ via a faithful irreducible algebraic representation $\tR:\sG\to\mathsf{GL}(\sV)$. We expect that these results are known in the community but we were unable to find an appropriate reference in the literature. 
	
	\begin{lemma}\label{lem.irred}
		Let $\sG$ be a real split semisimple Lie group, let $\sV$ be a finite dimensional vector space with $\dim\sV>1$, let $\tR:\sG\to\mathsf{GL}(\sV)$ be an irreducible algebraic representation, and let $X\in\sV$ with $X\neq0$. Then the additive group generated by $\{\tR_gX\mid g\in\sG\}\subset\sV$ is $\sV$.
	\end{lemma}
	\begin{proof}
		If possible, let us assume that $\tR_gX=X$ for all $g\in\sG$. Then $\tR(\sG)$ fixes the line $\R X$ and $\R X\subsetneq\sV$, a contradiction to the property that the representation $\tR$ is irreducible. 
		
		Hence we can assume that there exists a $g\in\sG$ such that $\tR_gX\neq X$. We use Lemma \ref{lem.ue} and the continuity of the action of $\sG$ to deduce that we can choose $g$ such that $g$ is loxodromic and the dimension of the unit eigenspace of $\tR_g$ is exactly $\dim \sV^0$. Let $m\in\sM$, $Z\in\fa^{++}$ and $h\in\sG$ be such that $g=hm\exp(Z)h^{-1}$. For all $\lambda\in\Omega\cup\{0\}$, let $Y_\lambda\in\sV^\lambda$ be such that \[Y:=\tR_h^{-1}X=Y_0+\sum_{\lambda\in\Omega}Y_\lambda.\]
		Moreover, as $\Omega$ is finite, we can slightly perturb $g$ and make sure that for all $\lambda,\nu\in\Omega$ we have $\lambda(Z)\neq\nu(Z)$ whenever $\lambda\neq\nu$. As $gX\neq X$ we obtain that $Y\neq Y_0$ and there exists $\mu\in\Omega$ such that $Y_\mu\neq0$. We observe that $\lambda(Z)\neq0$ for $\lambda\in\Omega$ and choose \[t_\lambda:=\left(1+\frac{\log 2}{\lambda(Z)}\right).\]
		It follows that $\tR_{\exp(t_\lambda Z)}Y_\lambda=2\tR_{\exp(Z)}Y_\lambda$. Indeed,
		\begin{align*}
			\tR_{\exp(t_\lambda Z)}Y_\lambda&=\exp(\lambda(t_\lambda Z))Y_\lambda=\exp(\log 2+\lambda(Z))Y_\lambda\\
			&=2\exp(\lambda(Z))Y_\lambda=2\tR_{\exp(Z)}Y_\lambda.
		\end{align*}
		Therefore, for all $\lambda\in\Omega$, we have $(2\tR_{\exp(Z)}-\tR_{\exp(t_\lambda Z)})Y_\lambda=0$. It follows that for all $\lambda\in(\Omega\cup\{0\})\setminus\{\mu\}$ and
		\[\tR^\mu:=\left(\tR_{\exp(Z)}-\tR_e\right)\prod_{\nu\in\Omega\setminus\{\mu\}}\left(2\tR_{\exp(Z)}-\tR_{\exp(t_\nu Z)}\right),\]
		we have $\tR^\mu Y_\lambda=0$. Therefore, we obtain that $\tR^\mu Y\in\sV^\mu$ and $\tR^\mu Y=\tR^\mu Y_\mu$. Moreover, we observe that
		\begin{align*}
			\tR^\mu Y_\mu&=\left(\tR_{\exp(Z)}-\tR_e\right)\prod_{\nu\in\Omega\setminus\{\mu\}}\left(2\tR_{\exp(Z)}-\tR_{\exp(t_\nu Z)}\right)Y_\mu\\
			&=(\exp(\mu(Z))-1)\prod_{\nu\in\Omega\setminus\{\mu\}}\left(2\exp(\mu(Z))-\exp(t_\nu \mu(Z))\right)Y_\mu.
		\end{align*}
		and $(\exp(\mu(Z))-1)\prod_{\nu\in\Omega\setminus\{\mu\}}\left(2\exp(\mu(Z))-\exp(t_\nu \mu(Z))\right)\neq0$. Hence
		\[\left\{\left(\tR_{\exp(tZ)}-\tR_{\exp(sZ)}\right)\tR^\mu Y\mid t,s\in\R\right\}=\R Y_\mu.\]
		Let $\mathsf{S}$ be the additive group generated by $\{\tR_gX\mid g\in\sG\}\subset\sV$, and hence, $\tR_h^{-1}X\in\mathsf{S}$. We observe that $\tR^\mu$ is inside the additive group generated by the set $\{\tR_g\mid g\in\sG\}\subset\mathfrak{gl}(\sV)$. It follows that $\R Y_\mu\subset\mathsf{S}$. Also, we observe that the additive group generated by $\tR(\sG)\R Y_\mu$ is the same as the vector space generated by $\tR(\sG)\R Y_\mu$. Moreover, the vector space generated by $\tR(\sG)\R Y_\mu$ is invariant under the action of $\tR(\sG)$. Using the irreducibility of the representation $\tR$, we obtain that $\tR(\sG)\R Y_\mu$ generates $\sV$. Therefore, we conclude that $\mathsf{S}=\sV$.
	\end{proof}
	
	\begin{proposition}\label{prop.normal}
		Let $\sG$ be a real split semisimple algebraic Lie group with trivial center, let $\sV$ be a finite dimensional vector space with $\dim\sV>1$. Let $\tR:\sG\to\mathsf{GL}(\sV)$ be a faithful irreducible algebraic representation and let $\sN$ be a normal subgroup of $\GV$. Then $\sN$ is either of the following subgroups: 
		\begin{enumerate}
			\item the trivial group,
			\item $\sG_i\ltimes\sV$, where $\sG_i$ is a normal subgroup of $\sG$.
		\end{enumerate}
	\end{proposition}
	\begin{proof}
		Let $\sN$ be a nontrivial normal subgroup of $\GV$. Then there exists $(g,X)\in\sN$ with $(g,X)\neq(e,0)$. Moreover, for any $(h,Y)\in\GV$ we observe that $(h,Y)^{-1}=(h^{-1},-\tR_h^{-1}Y)$ and hence
		\[(h,Y)(g,X)(h,Y)^{-1}=(hgh^{-1},Y+\tR_hX-\tR_{hgh^{-1}}Y).\]
		It follows that $(g,X)\in\sN$ if and only if $(hgh^{-1},Y+\tR_hX-\tR_{hgh^{-1}}Y)\in\sN$ for all $h\in\sG$ and $Y\in\sV$. 
		
		Now we consider the linear projection map $\tL:\GV\to\sG$ and observe that for all $h\in\sG$, $h\tL(\sN)h^{-1}\subset\tL(\sN)$. It follows that $\tL(\sN)$ is a normal subgroup of $\sG$. We prove our result in the following two parts:
		
		\begin{enumerate}
			\item $\tL(\sN)$ is trivial: As $\sN$ is nontrivial, in this case, we see that there exists $X\neq0$ such that $(e,X)\in\sN$. Hence, for all $h\in\sG$, we have 
			\[(h,0)(e,X)(h,0)^{-1}=(e,\tR_hX)\in\sN.\]
			As the representation $\tR$ is irreducible using Lemma \ref{lem.irred} we obtain that $(e,Y)\in\sN$ for all $Y\in\sV$. Therefore, we deduce that $\sN=\{e\}\ltimes\sV$.
			\item $\tL(\sN)$ is a nontrivial normal subgroup of $\sG$: In this case also we see that there exists $X\neq0$ such that $(e,X)\in\sN$. Indeed, if not then $\sN\cap\left(\{e\}\ltimes\sV\right)=\{(e,0)\}$ and hence
			\[\left.\tL\right|_{\sN}:\sN\to\sG\]
			is an isomorphism onto $\tL(\sN)$. It follows that for all $g\in\tL(\sN)$, there exists $X_g\in\sV$ such that $X_{gh}=X_g+\tR_gX_h$, and 
			\[\sN=\{(g,X_g)\mid g\in\tL(\sN)\}.\] 
			Since $\sN$ is normal inside $\GV$, for all $Y\in\sV$, we have
			\[(e,Y)(g,X_g)(e,Y)^{-1}=(g,Y+X_g-\tR_{g}Y)\in\sN.\]
			Hence $Y+X_g-\tR_{g}Y=X_g$ for all $Y\in\sV$, a contradiction. Therefore, there exists $X\neq0$ such that $(e,X)\in\sN$. Hence for all $h\in\sG$ we have 
			\[(h,0)(e,X)(h,0)^{-1}=(e,\tR_hX)\in\sN.\]
			As the representation $\tR$ is irreducible, using Lemma \ref{lem.irred}, we obtain that $(e,Y)\in\sN$ for all $Y\in\sV$ and we deduce that $\tL(\sN)\ltimes\sV\subset\sN$. Also, $\sN\subset\tL^{-1}(\tL(\sN))=\tL(\sN)\ltimes\sV$. It follows that $\tL(\sN)\ltimes\sV=\sN$
		\end{enumerate}
		Therefore, the only nontrivial normal subgroups of $\GV$ are of the form $\sG_i\ltimes\sV$ where $\sG_i$ is a normal subgroup of $\sG$.
	\end{proof}

	\bibliography{Library.bib}
	\bibliographystyle{alpha}
	
\end{document}